\documentclass[11pt]{article}
\usepackage[dvips]{graphics}
\usepackage{epsfig}
\usepackage{multirow}
\usepackage{array}
\usepackage{extarrows}
\usepackage{graphicx}
\usepackage{float}

\usepackage{comment}
\usepackage[titletoc,title]{appendix}
\numberwithin{equation}{section}
\usepackage{amsthm}
\usepackage{enumerate}

\usepackage{amssymb,amsmath}
\usepackage{latexsym}
\usepackage[usenames]{color}
\usepackage{pgf}

\usepackage{subfig}
\numberwithin{equation}{section}

\newcommand{\be}{\begin{equation}}
\newcommand{\ee}{\end{equation}}
\newcommand{\beaa}{\begin{eqnarray*}}
\newcommand{\eeaa}{\end{eqnarray*}}
\newcommand{\bea}{\begin{eqnarray}}
\newcommand{\eea}{\end{eqnarray}}
\newcommand{\lbl}{\label}

\newcommand{\bei}{\begin{itemize}}
\newcommand{\eei}{\end{itemize}}

\newcommand{\ml}{\mathcal}

\newtheorem{theorem}{ \noindent T{\footnotesize HEOREM}}

\newtheorem{lemma}{ \noindent L{\footnotesize EMMA}}[section]

\newtheorem{conjecture}{ \noindent C{\footnotesize ONJECTURE}}

\begin{document}

\title{Statistical Properties of  Eigenvalues of Laplace-Beltrami Operators}
\author{Tiefeng Jiang$^{1}$
and  Ke Wang$^2$\\
University of Minnesota \\and Hong Kong University of Science and Technology}

\date{}
\maketitle

\footnotetext[1]{School of Statistics, University of Minnesota, 224 Church
Street, S. E., MN55455, USA, jiang040@umn.edu. 
The research of Tiefeng Jiang was
supported in part by NSF Grant DMS-1209166 and DMS-1406279.}

\footnotetext[2]{Department of Mathematics, Hong Kong University of Science and Technology, Clear Water Bay, Kowloon, Hong Kong, kewang@ust.hk. Ke Wang was partially supported by Hong Kong RGC grant GRF 16301618, GRF 16308219 and ECS 26304920. 
}

\begin{abstract}
\noindent We study the eigenvalues of a Laplace-Beltrami operator defined on the set of the symmetric polynomials, where the eigenvalues are expressed in terms of partitions of integers. To study the behaviors of these eigenvalues, we assign partitions with the restricted uniform measure, the restricted Jack measure, the uniform measure or the Plancherel measure. We first obtain a new limit theorem on the restricted uniform measure. Then, by using it together with known results on other three measures, we prove that the global distribution of the eigenvalues is asymptotically a new distribution $\mu$, the Gamma distribution,  the Gumbel distribution and the Tracy-Widom distribution, respectively. The Tracy-Widom distribution is obtained for a special case only due to a technical constraint. An explicit representation of $\mu$ is obtained by a function of independent random variables.
Two open problems are also asked.
\\\\
\end{abstract}





\newpage

\section{Introduction}\lbl{Introduction}

Consider the Laplace-Beltrami operator
\bea\lbl{renxing_eqn}
\Delta_{\alpha}=\frac{\alpha}{2}\sum_{i=1}^my_i^2\frac{\partial^2}{\partial y_i^2} + \sum_{1\leq i\ne j\leq m}\frac{1}{y_i-y_j}\cdot y_i^2\frac{\partial}{\partial y_i}
\eea
defined on the set of symmetric and homogeneous polynomial $u(x_1, \cdots, x_m)$ of all degrees. There are two important quantities associated with the operator: its eigenfunctions and eigenvalues. The eigenfunctions are the $\alpha$-Jack polynomials and the eigenvalues are given by
\bea\lbl{Jingle}
\lambda_{\kappa}=n(m-1)+a(\kappa')\alpha-a(\kappa)
\eea
where $\kappa=(k_1, k_2, \cdots k_m)$ with $k_m>0$ is a partition of integer $n$, that is, $\sum_{i=1}^m k_i=n$ and $k_1\geq \cdots \geq k_m$, and $\kappa'$ is the transpose of $\kappa$ and
\bea\lbl{kernel_sea}
a(\kappa)=\sum_{i=1}^m (i-1) k_i=\sum_{i \ge 1} \binom{k_i'}{2};
\eea
see, for example, Theorem 3.1 from Stanley (1989) or p. 320 and p. 327 from Macdonald (1998).

The Jack polynomials are multivariate orthogonal polynomials (Macdonald, 1998). They consist of three special cases: the zonal polynomials with $\alpha=2$ which appear frequently in multivariate analysis of statistics (e.g., Muirhead, 1982); the Schur polynomials with $\alpha=1$ and the zonal spherical functions with $\alpha=\frac{1}{2}$ which have rich applications in the group representation theory, algebraic combinatorics, statistics and random matrix theory [e.g., Macdonald (1998), Fulton and Harris (1999), Forrester (2010)].

In this paper we consider the statistical behaviors of the eigenvalues $\lambda_{\kappa}$ given in (\ref{Jingle}). That is,  how does $\lambda_{\kappa}$ look like if $\kappa$ is picked randomly? For example, what are the sample mean and the sample variance of $\lambda_{\kappa}$'s, respectively? In fact, even though the expression of $\lambda_{\kappa}$ is explicit, it is non-trivial to answer the question. In particular, it is hard to use a software to analyze them because the size of $\{\kappa;\, \kappa\ \mbox{is a partition of}\ n \}$ is of order $\frac{1}{n}e^{C\sqrt{n}}$ for some constant $C$; see (\ref{Raman}).

The same question was asked for the eigenvalues of random matrices and the eigenvalues of Laplace operators defined on compact Riemannian manifolds. For instance, the typical behavior of the  eigenvalues of a large Wigner matrix is the Wigner semi-circle law (Wigner, 1958), and that of a Wishart matrix is the Marchenko-Pastur law (Marchenko and Pastur, 1967). The Weyl law is obtained for the eigenvalues of a Laplace-Beltrami operator acting on functions with the Dirichlet condition which  vanish at the boundary of a bounded domain in the Euclidean  space  (Weyl, 1911). For example, the Weyl asymptotic formula says that $\frac{\lambda_k}{k^{d/2}} \sim (4\pi)^{-d/2}\frac{vol(M)}{\Gamma(\frac{d}{2}+1)}$
as $k\to\infty$, where $d$ is the dimension of $M$ and $vol(M)$ is the volume of $M$. It is proved by analyzing the trace of a heat kernel; see, e.g., p. 13 from Borthwick (2012). Let $\Delta_S$ be  the spherical Laplacian operator on the unit sphere  in $\mathbb{R}^{n+1}.$ It is known that the eigenvalues of $-\Delta_S$ are $k(k+n-1)$ for $k=0,1,2,\cdots$ with multiplicity of $\binom{n+k}{n}- \binom{n+k-2}{n}$;  see, e.g., ch. 2 from Shubin (2001). Some other types of Laplace-Beltrami operators appear in the Riemannian symmetric spaces; see, e.g.,  M\'{e}liot (2014). Their eigenvalues are also expressed in terms of partitions of integers. Similar to  this paper, those eigenvalues can also be analyzed.

To study a typical property of $\lambda_{\kappa}$ in (\ref{Jingle}), how do we pick a partition randomly?  We will sample $\kappa$  by using  four popular probability measures:  the restricted uniform measure, the restricted Jack measure, the uniform measure and  the Plancherel measure. While studying $\lambda_{\kappa}$  for fixed operator $\Delta_{\alpha}$ with $m$ variables, the two restricted measures are adopted to investigate $\lambda_{\kappa}$ by letting $n$ become large. Look at the infinite version of the operator $\Delta_{\alpha}$:
\bea\lbl{renxing_eqn2}
\Delta_{\alpha,\infty}:=\frac{\alpha}{2}\sum_{i=1}^{\infty}y_i^2\frac{\partial^2}{\partial y_i^2} + \sum_{1\leq i\ne j<  \infty}\frac{1}{y_i-y_j}\cdot y_i^2\frac{\partial}{\partial y_i},
\eea
which acts on the set of symmetric and homogeneous polynomial $u(x_1, \cdots, x_m)$ of degree $m\geq 0$ being arbitrary; see, for example, page 327 from Macdonald (1998). Recall (\ref{Jingle}). At ``level" $n$, the set of eigenvalues of $\Delta_{\alpha,\infty}$ is $\{\lambda_{\kappa}; \kappa \in \ml{P}_n\}$. In this situation, the partition length $m$ depends on $n$, this is the reason that we employ  the uniform measure and  the Plancherel measure.

Under the four measures, we prove in this paper that the limiting distribution of random variable $\lambda_{\kappa}$ is a new distribution $\mu$, the Gamma distribution, the Gumbel distribution and the Tracy-Widom distribution, respectively. Due to a technical constraint, the Tracy-Widom distribution is obtained for the case $\alpha=1$ only. For other $\alpha>0$, see a less precise result in Theorem \ref{thm:LLLplan} and Conjecture 1. The distribution $\mu$ is characterized by a function of independent random variables. More specifically, $\mu$ is the push-forward of  $\frac{\alpha}{2}\cdot \frac{\xi_1^2+\cdots + \xi_m^2}{(\xi_1+\cdots + \xi_m)^2}$ where $\xi_i$'s are i.i.d. random variables with the density  $e^{-x}I(x\geq 0).$ In the following we will present these  results in this order. We will see, in addition to a tool on random partitions  developed in this paper (Theorem \ref{finite_theorem}),  a fruitful of work along this direction has been used: 
 the approximation result on random partitions under the uniform measure by Pittel (1997);  the largest part of a random partition asymptotically following the Tracy-Widom law by Baik {\it et al}. (1999), Borodin {\it et al}.
(2000), Okounkov (2000) and Johannson (2001); Kerov's central limit theorem (Ivanov and Olshanski, 2001); the Stein method on random partitions by Fulman (2004); the limit law of random partitions under restricted Jack measure by Matsumoto (2008).

A consequence of our theory provides an answer at (\ref{mean_variance}) for the size of the sample mean and sample variance of $\lambda_{\kappa}$ aforementioned.

We do not pursue applications of our results in this paper. They may be useful in Migdal's formula for the partition functions of the 2D Yang-Mills theory [e.g., Witten (1991) and  Woodward (2005)]. Further possibilities can be seen, e.g.,  in the papers by Okounkov (2003) and Borodin and Gorin (2012).

We study the eigenvalues of the Laplace-Beltrami operator in terms of four different measures. This  can also be continued by other probability measures on random partitions, for example,  the $q$-analog of the Plancherel measure [e.g., Kerov (1992) and F\'{e}ray and M\'{e}liot (2012)], the multiplicative measures [e.g., Vershik (1996)], the $\beta$-Plancherel measure (Baik and Rains, 2001), the Jack  measure and the Schur measure [e.g., Okounkov (2003)].

{\bf Organization of the paper:} We present our limit laws by using  the four measures in Sections \ref{sec:restricted-uniform}, \ref{sec:restricted-Jack}, \ref{sec:uniform} and \ref{sec:Plancherel}, respectively. Four figures corresponding to the four theorems are provided to show that curves based on data and the limiting curves match very well. In Section \ref{sec:New_Random}, we state a new result on random partitions.  In Section \ref{main:proofs}, we prove all of the results. In Section \ref{appendix:last} (Appendix), we compute the sample mean and sample variance of $\lambda_{\kappa}$ mentioned in \eqref{mean_variance}, calculate a non-trivial  integral used earlier and derive the density function in Theorem \ref{cancel_temple} for two cases.


{\bf Notation:} $f(n) \sim g(n)$ if $\lim_{n \to \infty} f(n)/g(n) =1$. We
assume that $n$ is large and asymptotic notation such as $o(\cdot), O(\cdot)$ will be used under the assumption that $n\to \infty$. Let $\{X_n;\, n\geq 1\}$ be random variables and $\{w_n;\, n\geq 1\}$ be non-zero constants. If $\{X_n/w_n;\, n\geq 1\}$ is bounded in probability, i.e., $\lim_{K\to\infty}\sup_{n\geq 1}P(|X_n/w_n|\geq K)=0$, we then write $X_n=O_p(w_n)$ as $n\to\infty.$ If $X_n/w_n$ converges to $0$ in probability, we write $X_n=o_p(w_n)$. We write ``cdf" for ``cumulative distribution function" and ``pdf" for ``probability density function". We use $\kappa \vdash n$ if $\kappa$ is a partition of $n$. The notation $[x]$ stands for the largest integer less than or equal to $x$.

{\bf Graphs:} The convergence in Theorems \ref{cancel_temple}, \ref{Gamma_surprise}, \ref{vase_flower} and \ref{difficult_easy} are illustrated in Figures \ref{fig:ru}-\ref{fig:plancherel}: we compare the empirical pdfs, also called histograms in statistics literature, with their limiting pdfs in the left columns. The right columns compare the empirical cdfs with their limiting cdfs. These graphs suggest that the empirical ones and their limits match very well.

\subsection{Limit under restricted uniform distribution}\label{sec:restricted-uniform}
Let $\ml{P}_n$ denote the set of all partitions of $n$. Now we consider a subset of $\ml{P}_n$. Let $\ml{P}_n(m)$ and $\ml{P}_n'(m)$ be the sets of partitions of $n$ with lengths at most $m$ and with lengths exactly equal to $m$, respectively. 
Note that $\ml{P}_n(n)=\ml{P}_n$. Our limiting laws of $\lambda_\kappa$ under the two measures are derived as follows. A simulation is shown in Figure \ref{fig:ru}.

\begin{theorem}\lbl{cancel_temple} Let $\kappa \vdash n$ and $\lambda_{\kappa}$ be as in (\ref{Jingle}) with $\alpha >0$. Let $m\geq 2$, $\{\xi_i;\, 1\leq i \leq m\}$ be i.i.d. random variables with density  $e^{-x}I(x\geq 0)$ and $\mu$ be the measure induced by $\frac{\alpha}{2}\cdot \frac{\xi_1^2+\cdots + \xi_m^2}{(\xi_1+\cdots + \xi_m)^2}$. Then,
under the uniform measure on $\ml{P}_n(m)$ or $\ml{P}_n'(m)$, $\frac{\lambda_{\kappa}}{n^2}\to \mu$ weakly  as $n\to\infty$.
\end{theorem}
By the definition of $\ml{P}_n'(m)$, the above theorem gives the typical behavior of the eigenvalues of the Laplace-Beltrami operator for fixed $m$. We will prove this theorem in Section \ref{sec:proof:restricted-uniform}. In  Section \ref{appendix:integral}, we compute  the pdf $f(t)$ of $\frac{\xi_1^2+\cdots + \xi_m^2}{(\xi_1+\cdots + \xi_m)^2}$, which is different from $\mu$ by a multiplicative scalar, for $m=2, 3$. It shows that $f(t)= \frac{1}{\sqrt{2t-1}}I_{[\frac{1}{2}, 1]}(t)$ for $m=2$; for $m=3$, the support of $\mu$ is  $[\frac{1}{3}, 1]$ and
\beaa
 f(t)=
 \begin{cases}
   \frac{2}{\sqrt{3}} \pi, & \text{if } \frac{1}{3} \le t < \frac{1}{2}; \\
   \frac{2}{\sqrt{3}} \big( \pi - 3\arccos\frac{1}{\sqrt{6t-2}} \big),      & \text{if } \frac{1}{2} \le t \le 1.\\
  \end{cases}
\eeaa
From our computation, it does not seem easy to derive an explicit formula for the density function as $m\geq 4$. It would be interesting to explore this.
The proof of Theorem  \ref{cancel_temple} relies on a new result on random partitions from $\ml{P}_n(m)$ and  $\ml{P}_n'(m)$ with the uniform distributions, which is of independent interest. We postpone it until Section \ref{sec:New_Random}.


Given numbers $x_1, \cdots, x_r$, the average and dispersion/fluctation of the data are usually measured by  the sample mean $\bar{x}$ and the sample variance $s^2$, respectively, where
\bea\lbl{pro_land}
\bar{x}= \frac{1}{r}\sum_{i=1}^rx_i\ \ \mbox{and}\ \ s^2=\frac{1}{r-1}\sum_{i=1}^r(x_i-\bar{x})^2.
\eea
Replacing $x_i$'s by $\lambda_{\kappa}$'s as in (\ref{Jingle}) for all $\kappa\in \mathcal{P}_n(m)'$, then $r=|\mathcal{P}_n(m)'|$. We will prove in Section 3.1 that, by Theorem 1 and the bounded convergence theorem, we have
\bea\lbl{mean_variance}
\frac{\bar{x}}{n^2}\to\frac{\alpha}{m+1}\ \ \mbox{and} \ \ \frac{s^2}{n^4}\to \frac{(m-1)\alpha^2}{(m+1)^2(m+2)(m+3)}
\eea
as $n\to\infty$.  The proof is given in Section \ref{appendix:mean_variance}. The moment $(1/r)\sum_{i=1}^rx_i^j$ with $x_i$'s replaced by $\lambda_{\kappa}$'s can be analyzed similarly for other $j\geq 3.$

\medskip

\noindent\textbf{Comments}.  By a standard characterization of spacings of i.i.d. random variables with the uniform distribution on $[0, 1]$ through exponential random variables [see, e.g.,  Sec 2.5.3 from  Rubinstein and Kroese (2007) and Chapter 5 from  Devroye (1986)], the limiting distribution $\mu$ in Theorem \ref{cancel_temple} is identical to any of the following:



\noindent (i) $\frac{\alpha}{2}\cdot \sum_{i=1}^m y_i^2$, where $y:=(y_1,\ldots,y_m)$ uniformly sits on   $\{y \in [0,1]^{m}; \sum_{i=1}^{m}y_i = 1 \}$.

\noindent (ii) $\frac{\alpha}{2}\cdot \sum_{i=1}^m (U_{(i)}-U_{(i-1)})^2$ where $U_{(1)} \le \ldots \le U_{(m-1)}$ are the order statistics of i.i.d. random variables $\{U_i;\, 1\leq i \leq m\}$ with uniform distribution on $[0,1]$ and $U_{(0)}=0$, $U_{(m)}=1$.



\begin{figure}[!ht]
 \begin{center}
   \includegraphics[width=10cm]{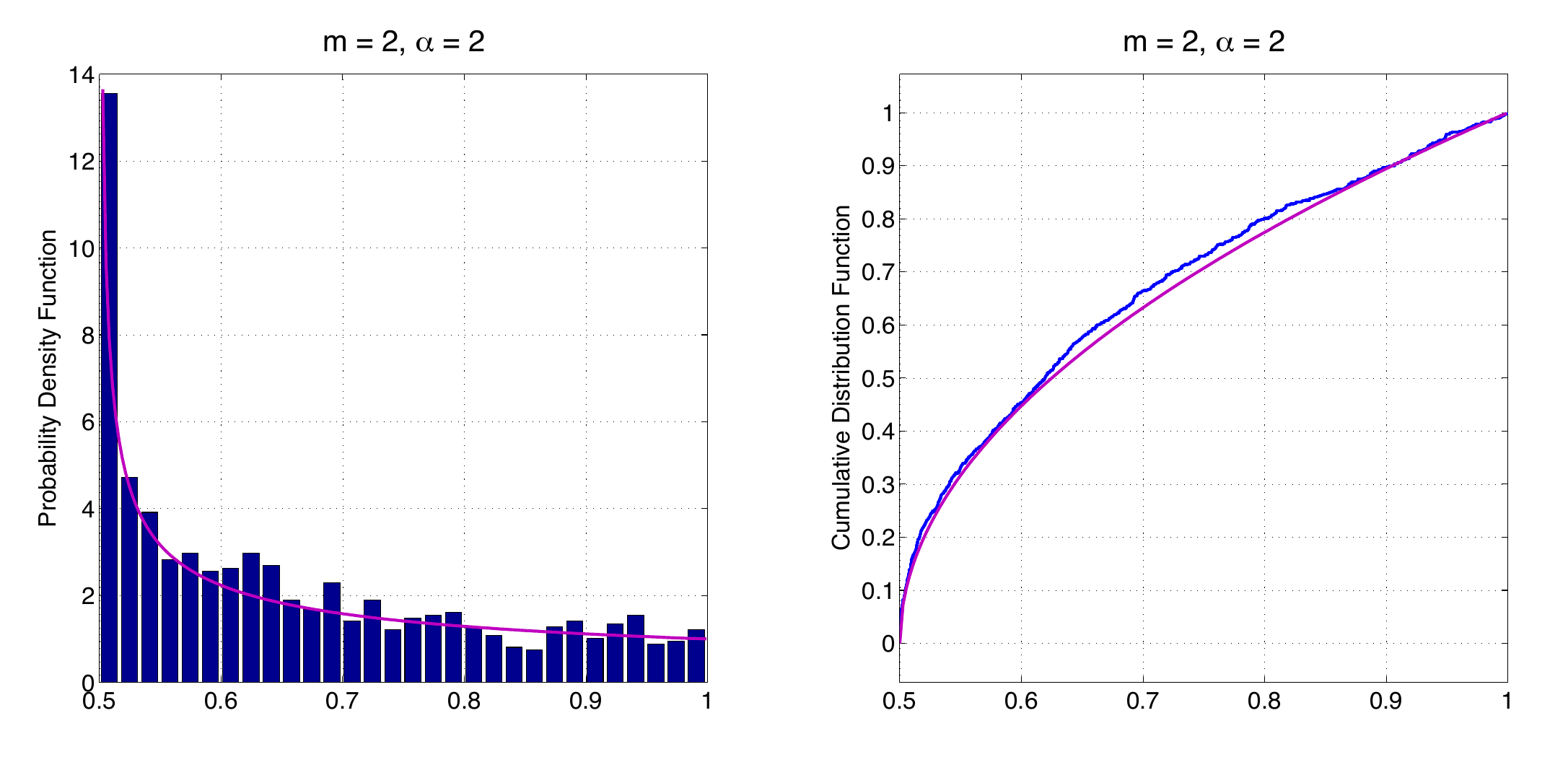}
   \caption{The histogram/empirical cdf   of $\lambda_\kappa/n^2$ for $\alpha=m=2$ is compared with pdf/cdf of $\mu$ in  Theorem \ref{cancel_temple} at $n=2000$. We independently sampled 1000 points according to $\mu$. }
   \label{fig:ru}
 \end{center}
\end{figure}

\subsection{Limit under restricted Jack distribution}\label{sec:restricted-Jack}

\noindent 
The Jack measure with parameter $\alpha$ chooses a partition
$\kappa \in \ml{P}_n$ with probability
\begin{equation}\label{eq:jack}
 P(\kappa) = \frac{\alpha^n n! }{c_{\kappa}(\alpha) c'_{\kappa} (\alpha)},
\end{equation}
where $$c_{\kappa}(\alpha) = \prod_{(i,j) \in \kappa} (\alpha (\kappa_i - j) + (\kappa'_j -i) + 1) \quad \text{and} \quad c'_{\kappa}(\alpha) = \prod_{(i,j) \in \kappa} (\alpha (\kappa_i - j) + (\kappa'_j -i) + \alpha).$$
The Jack measure naturally appears in the Atiyah-Bott formula from the algebraic geometry; see an elaboration in the notes by Okounkov (2013). 

In this section, we consider the random restricted Jack measure studied by Matsumoto (2008). Let $m$ be a fixed positive integer. Recall $\ml{P}_n(m)$ is the set of integer partitions of $n$ with at most $m$ parts.
The induced restricted Jack distribution with parameter $\alpha$ on $\ml{P}_n(m)$ is defined by [we follow the notation by Matsumoto (2008)]
\begin{equation}\label{eq:restrictedjack}
 P_{n,m}^{\alpha}(\kappa) = \frac{1}{C_{n,m}(\alpha)} \frac{1}{c_{\kappa}(\alpha) c'_{\kappa}(\alpha)}, \quad \kappa \in \ml{P}_n(m),
\end{equation}
with the normalizing constant
$$C_{n,m}(\alpha) = \sum_{\mu \in \ml{P}_n(m)} \frac{1}{c_{\mu}(\alpha) c'_{\mu}(\alpha)}.$$

\begin{figure}[!ht]
 \begin{center}
   \includegraphics[width=10cm]{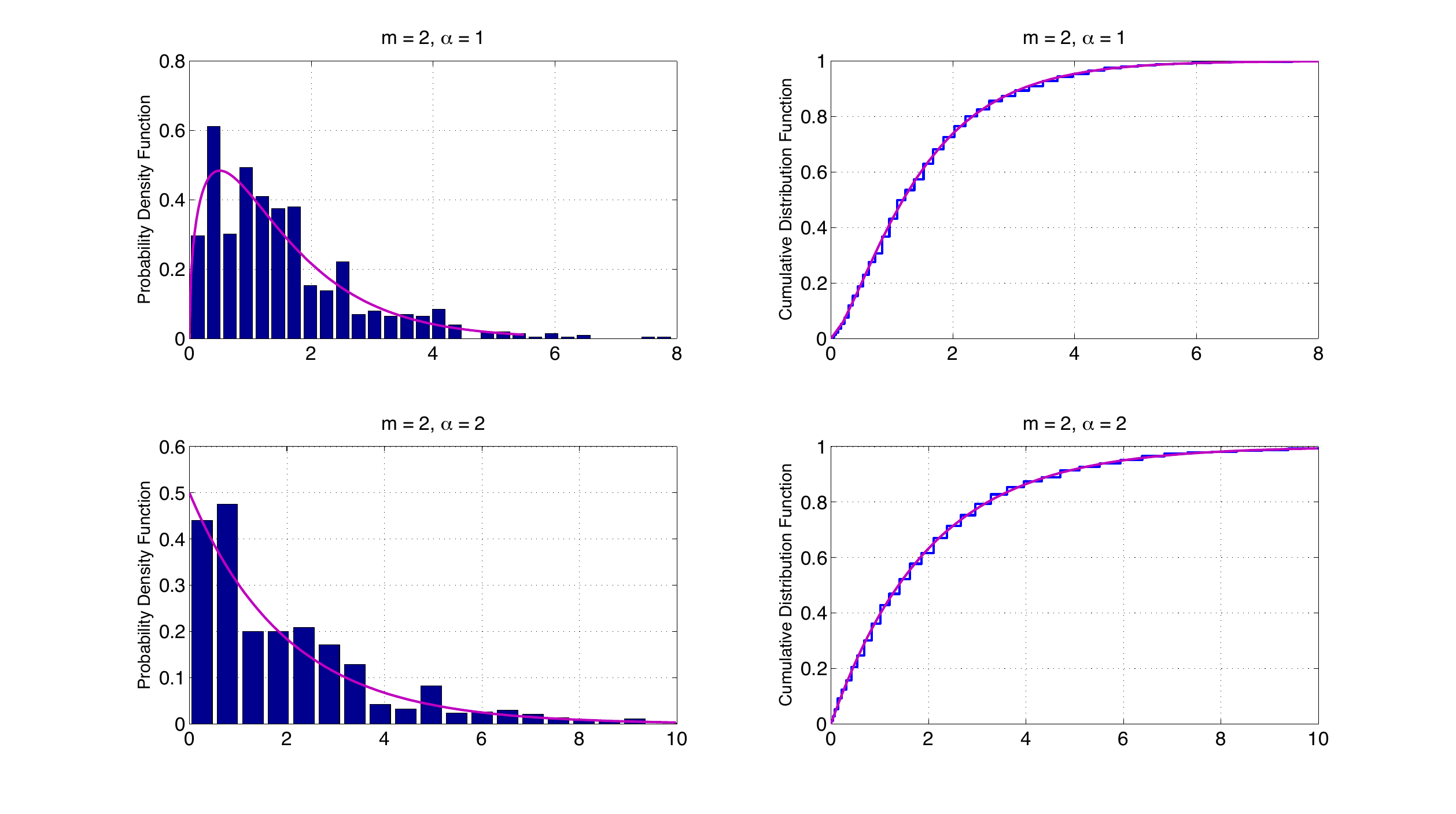}
   \caption{Top row compares histogram/empirical cdf of $(\lambda_{n}-a_n)/b_n$ in Theorem \ref{Gamma_surprise} for $m=2,\, \alpha=1$ with Gamma pdf/cdf at $n=1000$. The quantity ``$(\lambda_{n}-a_n)/b_n$" is independently sampled for $800$ times. Similar interpretation applies to the bottom row for $m=\alpha=2.$}
   \label{fig:rj}
 \end{center}
\end{figure}
Similarly, replacing  $\ml{P}_n(m)$  above with ``$\ml{P}_n'(m)$", we get the restricted Jack measure on $\ml{P}_n'(m)$. We call it  $Q_{n,m}^{\alpha}$. The following is our result under the two measures.
\begin{theorem}\lbl{Gamma_surprise}
Let $\kappa \vdash n$ and $\lambda_{\kappa}$ be as in (\ref{Jingle}) with parameter $\alpha >0$. Set $\beta=2/\alpha$. Then, for given $m\geq 2$, if $\kappa$ is chosen according to $P_{n,m}^{\alpha}$ or $Q_{n,m}^{\alpha}$, then
$$\frac{\lambda_\kappa-a_n}{b_n} \to \text{Gamma distribution with pdf } h(x)=\frac{1}{\Gamma(v)\, (2/\beta)^{v}}x^{v-1}e^{-\beta x/2} \text{ for } x\geq 0$$ weakly as $n\to \infty$, where
\beaa
a_n=\frac{m-\alpha-1}{2}n+\frac{\alpha}{2m}n^2,\quad  b_n=\frac{n}{2m},\quad  v=\frac{1}{4}(m-1)\cdot(m\beta + 2).
\eeaa
\end{theorem}
By the definition of $\ml{P}_n'(m)$, the above theorem gives the typical behavior of the eigenvalues of the Laplace-Beltrami operator for fixed $m$ under the restricted Jack measure.

Write $v=\frac{1}{2}\cdot\frac{1}{2}(m-1)(m\beta+2)$. Then the limiting distribution becomes a $\chi^2$ distribution with (integer) degree of freedom $\frac{1}{2}(m-1)(m\beta+2)$ for $\beta=1,2$ or $4$.  See Figure \ref{fig:rj} for numerical simulation. 

We will prove Theorem \ref{Gamma_surprise} in Section \ref{sec:proof:restricted-Jack}. Indeed, since $\mathcal{P}_m(n)$ and $\mathcal{P}_m'(n)$ have asymptotically the same size, and neither the uniform measure nor the restricted Jack measure is concentrated on any set in $\mathcal P_m(n)$ or $\mathcal P_m'(n)$, for the proofs of Theorem \ref{cancel_temple} and \ref{Gamma_surprise}, it suffices to prove the results on $\mathcal{P}_m(n)$.


\subsection{Limit under uniform distribution}\label{sec:uniform}

Let $\ml{P}_n$ denote the set of all partitions of $n$ and $p(n)$ the number of such partitions. Recall the operator $\Delta_{\alpha, \infty}$ in (\ref{renxing_eqn2}) and the eigenvalues in (\ref{Jingle}). At ``level" $n$, the set of eigenvalues is $\{\lambda_{\kappa}; \kappa \in \ml{P}_n\}$. The parameter ``$m$" appearing in Theorems \ref{cancel_temple} and \ref{Gamma_surprise} is irrelevant here. Now we choose $\kappa$ according to the uniform distribution on $\ml{P}_n$.
The limiting distribution of $\lambda_{\kappa}$ is given below. Denote $\zeta(x)$  the Riemman's zeta function.
\begin{theorem}\lbl{vase_flower}
Let  $\kappa \vdash n$ and $\lambda_{\kappa}$ be as in (\ref{Jingle}) with parameter $\alpha >0$. If $\kappa$ is chosen uniformly from the set $\ml{P}_n$, then
$$cn^{-3/2}\lambda_{\kappa} -\log \frac{\sqrt{n}}{c} \to \text{Gumbel distribution with cdf } G(x)=\exp\big(-e^{-(x+K)}\big) $$
weakly as $n \to \infty$, where $c=\frac{\pi}{\sqrt{6}}$ and $K=\frac{6\zeta(3)}{\pi^2}(1-\alpha)$.
\end{theorem}

In Figure \ref{fig:uniform}, we simulate the distribution of $\lambda_\kappa$ at $n=4000$ and compare with the Gumbel distribution $G(x)$ as in Theorem \ref{vase_flower}. Its proof will be given at Section \ref{proof_vase_flower}. Comparing Figure \ref{fig:ru} and Figure \ref{fig:uniform}, we see the limiting behaviours of $\lambda_\kappa$ differ significantly under uniform measures on $\mathcal{P}_n(m)$ with $m$ fixed and $\mathcal{P}_n(n)$ with $m=n\to \infty$ respectively. 

\begin{figure}[!ht]
 \begin{center}
   \includegraphics[width=10cm]{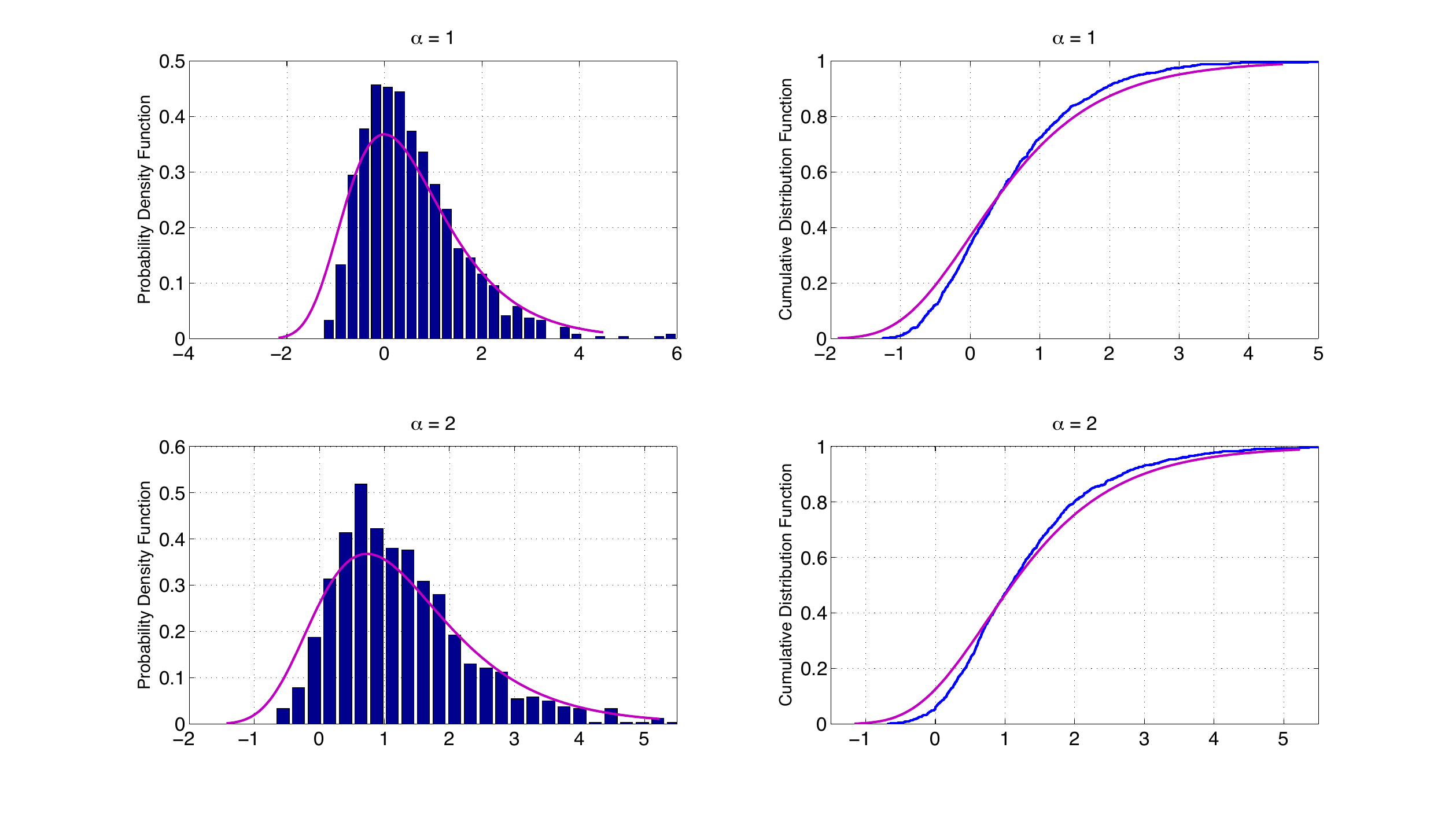}
   \caption{Top row compares histogram/empirical cdf of ``$cn^{-3/2}\lambda_{\kappa} -\log \frac{\sqrt{n}}{c}$" for $\alpha=1$ with the pdf $G'(x)$/cdf $G(x)$ in Theorem \ref{vase_flower} at $n=4000$. The quantity ``$cn^{-3/2}\lambda_{\kappa} -\log \frac{\sqrt{n}}{c}$" is independently sampled for $1000$ times. Similar interpretation applies to the bottom row for $\alpha=2.$}
   \label{fig:uniform}
 \end{center}
\end{figure}

\subsection{Limit under Plancherel distribution}\label{sec:Plancherel}

Review the operator $\Delta_{\alpha, \infty}$ in (\ref{renxing_eqn2}) and the eigenvalues in (\ref{Jingle}). At ``level" $n$, the set of eigenvalues is $\{\lambda_{\kappa}; \kappa \in \ml{P}_n\}$. There is no parameter ``$m$" appearing in Theorems \ref{cancel_temple} and \ref{Gamma_surprise}. We now apply the Plancherel measure to understand this set of eigenvalues.

A random partition $\kappa$ of $n$ has the Plancherel measure if it is chosen from $\ml{P}_n$ with probability
\begin{equation}\label{eq:plan}
P(\kappa) = \frac{\dim(\kappa)^2}{n!},
\end{equation}
where $\dim(\kappa)$ is the dimension of irreducible representations of the symmetric group $\ml{S}_n$ associated with $\kappa$. It is given by
$$\dim(\kappa) = \frac{n!}{\prod_{(i,j) \in \kappa} (k_i - j +k'_j -i + 1)}.$$
See, e.g., Frame {\it et al}. (1954). This measure is a special case of the $\alpha$-Jack measure defined in (\ref{eq:jack}) with $\alpha=1$. The Tracy-Widom distribution is defined by
\begin{eqnarray}\label{tai}
F_2(s)=\exp\left(-\int_{s}^{\infty}(x-s)q(x)^2\,dx\right),\
s\in
  \mathbb{R},
\end{eqnarray}
where $q(x)$ is the solution to the Painl\'eve II  differential equation
\begin{eqnarray*}
& & q''(x)=xq(x)+2q(x)^3\ \ \mbox{with boundary condition}\\
& & q(x) \sim \mbox{Ai}(x)\ \mbox{as}\ x\to +\infty
\end{eqnarray*}
and $\mbox{Ai}(x)$ denotes the Airy function. Replacing the uniform measure in Theorem \ref{vase_flower} with the Plancherel measure, we get the following result.

\begin{theorem}\lbl{difficult_easy} Let  $\kappa \vdash n$ and $\lambda_{\kappa}$ be as in (\ref{Jingle}) with parameter $\alpha=1$. If $\kappa$ follows the Plancherel  measure, then
\beaa
\frac{\lambda_{\kappa} - 2 \cdot n^{3/2}}{n^{7/6} } \to F_2
\eeaa
weakly as $n\to\infty$, where $F_2$ is as in (\ref{tai}).
\end{theorem}
\begin{figure}[!ht]
 \begin{center}
   \includegraphics[width=10cm]{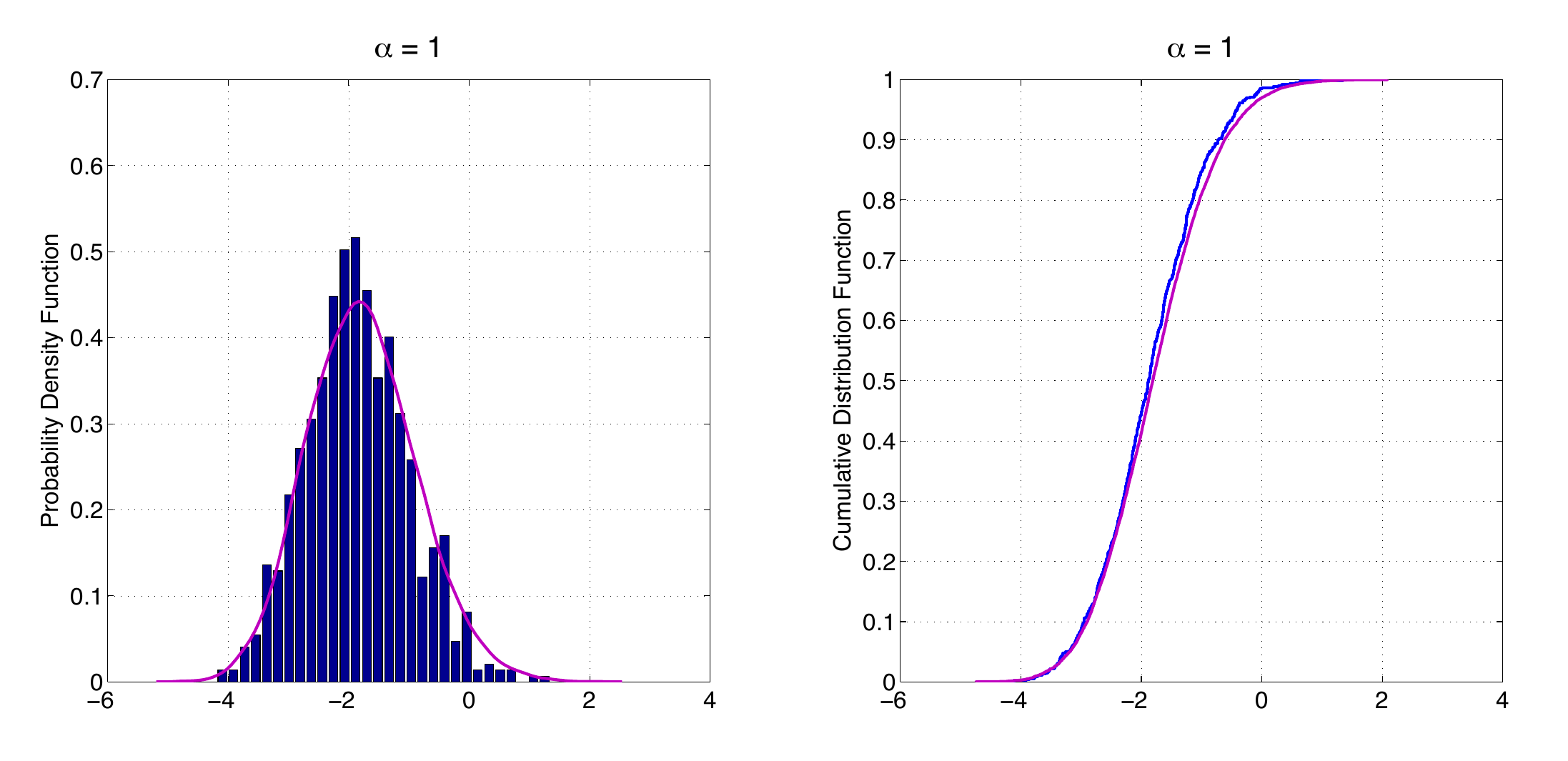}
   \caption{The histogram/empirical cdf   of $T:=(\lambda_{\kappa} - 2 \cdot n^{3/2})n^{-7/6}$ for $\alpha=1$ is compared with pdf/cdf of $F_2$  in  Theorem \ref{difficult_easy} at $n=5000$. The value of $T$ is independently sampled  for $800$ times.}
   \label{fig:plancherel}
 \end{center}
\end{figure}

The proof of this theorem  will be presented in Section \ref{Proof_difficult_easy}. In Figure \ref{fig:plancherel}, we simulate the limiting distribution of $\lambda_\kappa$ with $\alpha=1$ and compare it with $F_2$. For any $\alpha\ne 1$, we prove a weak result as follows.
\begin{theorem}\label{thm:LLLplan}
Let  $\kappa \vdash n$ and $\lambda_{\kappa}$ be as in (\ref{Jingle}) with parameter $\alpha>0$. If $\kappa$ follows the Plancherel measure, then for any sequence of real numbers $\{a_n > 0\}$ with $\lim_{n\to \infty}a_n = \infty$,
$$\frac{\lambda_{\kappa} -\left(2+ \frac{128}{27\pi^2} (\alpha-1) \right) n^{3/2} }{ n^{5/4} \cdot a_n } \to 0$$
in probability as $n\to\infty$.
\end{theorem}
The proof of Theorem \ref{thm:LLLplan} will be given in Section \ref{non_stop_thousand}. We provide a conjecture on the limiting distribution for $\lambda_{\kappa}$ with arbitrary $\alpha>0$ under Plancherel measure.

\begin{conjecture}
Let  $\kappa \vdash n$ and $\lambda_{\kappa}$ be as in (\ref{Jingle}). If $\kappa$ has the Plancherel measure, then
\beaa
\frac{\lambda_{\kappa} - \left(2 + \frac{128}{27\pi^2}(\alpha-1) \right)\cdot n^{3/2}}{n^{7/6} } \to (3-2\alpha)F_2
\eeaa
weakly as $n\to\infty$, where $F_2$ is as in (\ref{tai}).
\end{conjecture}
The quantities ``$3-2\alpha$" and ``$n^{7/6}$" can be seen from the proofs of Theorems \ref{difficult_easy} and \ref{thm:LLLplan}. The conjecture will be confirmed if there is a stronger version of the central limit theorem by Kerov [Theorem 5.5 by Ivanov and Olshanski (2001)]: the central limit theorem still holds if the Chebyshev polynomials are  replaced by smooth functions.

One can also consider the same quantity under the $\alpha$-Jack  measure as in (\ref{eq:jack}), a generalization of the Plancherel measure. However, under this measure, the limiting distribution of the largest part of a random partition is not known. There is only a conjecture made by  Dolega and F\'{e}ray (2014). In virtue of this and our proof of Theorem \ref{difficult_easy}, we give a conjecture on $\lambda_{\kappa}$ studied in this paper. 

\begin{conjecture}
Let  $\kappa \vdash n$ and $\lambda_{\kappa}$ be as in (\ref{Jingle}) with parameter $\alpha>0$. If $\kappa$ follows the $\alpha$-Jack  measure [the ``$\alpha$" here is the same as that in (\ref{Jingle})], then
\beaa
\frac{\lambda_{\kappa} - 2\alpha^{-1/2}n^{3/2} }{n^{7/6} } \to F_{\alpha}
\eeaa
weakly as $n\to\infty$, and $F_{\alpha}$ is the $\alpha$-analogue of the Tracy-Widom distribution $F_2$ in (\ref{tai}). The law $F_{\alpha}$ is equal to $\Lambda_0$ stated in Theorem 1.1 from Ram\'{\i}rez {\it et al}. (2011).
\end{conjecture}

\subsection{A new result on random partitions}\label{sec:New_Random}

At the same time as proving Theorem \ref{cancel_temple}, we find the following  result on the restricted random partitions, which is also interesting on its own merits.

\begin{theorem}\lbl{finite_theorem} Given $m\geq 2$. Let $\ml{P}_n(m)$ and $\ml{P}_n(m)'$ be as in Theorem \ref{cancel_temple}. Let $(k_1, \cdots, k_m)\vdash n$ follow the uniform distribution on $\ml{P}_n(m)$ or $\ml{P}_n(m)'$. Then, as $n\to\infty$, $\frac{1}{n}(k_1, \cdots, k_m)$ converges weakly to  the uniform distribution on the ordered simplex
\bea\lbl{unifDelta}
\Delta:=\Big\{(x_1, \cdots, x_{m})\in [0,1]^{m};\,  x_1>\cdots >x_m\ \mbox{and}\ \sum_{i=1}^{m} x_i=1\Big\}.
\eea
\end{theorem}

It is known from Rabinowitz (1989) that the volume of $\Delta=\frac{\sqrt{m}}{m!(m-1)!}.$ So the density function of the uniform distribution on $\Delta$ is equal to $\frac{m!(m-1)!}{\sqrt{m}}.$

If one picks a random partition $\kappa=(k_1, k_2,\cdots)\vdash n$ under the uniform measure, that is, under the uniform measure on $\ml{P}_n$, put the Young diagram of $\kappa$ in the first quadrant, and shrink the curve by a factor of $n^{-1/2}$, Vershik (1996) proves that the new random curve converges to the curve $e^{-cx} + e^{-cy}=1$ for $x, y>0$, where $c=\pi/\sqrt{6}.$ For the Plancherel measure, Logan and Shepp (1977) and  Vershik and Kerov (1977) prove that, for a rotated and shrunk Young diagram $\kappa$, its boundary curve (see the ``zig-zag" curve in Figure \ref{fig:kerov})  converges to
$\Omega(x)$, where
\bea\lbl{jilin}
\Omega(x)=\begin{cases}
\frac{2}{\pi}(x\arcsin\frac{x}{2} + \sqrt{4-x^2}),& \text{$|x|\leq 2$};\\
|x|, & \text{$|x|>2$}.
\end{cases}
\eea
As $m$ is no longer fixed but equal to $n$, the above law differs from the one presented in Theorem 6.  We will prove this result  in Section \ref{proof_finite_theorem}.

\section{Proofs}\lbl{main:proofs}
In this section we will prove the theorems stated earlier. Theorem \ref{finite_theorem} will be proved first because it  will be used later.

\subsection{Proof of Theorem \ref{finite_theorem}}\lbl{proof_finite_theorem}
The following conclusion is based on the fact that $\mathcal{P}_n(m)$ and $\mathcal{P}_n(m)'$ have asymptotically the same size, and is not difficult to prove. We skip its proof.
\begin{lemma}\lbl{Carnige} Review the notation in Theorem \ref{finite_theorem}. Assume,  under $\ml{P}_n(m)$,   $\frac{1}{n}(k_1, \cdots, k_m)$ converges weakly to  the uniform distribution on $\Delta$ as $n\to\infty$.
Then the same convergence also holds true under $\ml{P}_n(m)'$.
\end{lemma}

We now introduce the  equivalence of two uniform distributions.
\begin{lemma}\lbl{paint_swim} Let $m\geq 2$ and $X_1> \cdots > X_m\geq 0$ be random variables. Recall (\ref{unifDelta}).  Set
\bea\lbl{thunder_storm}
W=\Big\{(x_1, \cdots, x_{m-1})\in [0, 1]^{m-1};\, x_1>\cdots  >x_m\geq 0\ \mbox{and}\  \sum_{i=1}^{m}x_i=1 \Big\}.
\eea
Then $(X_1, \cdots, X_m)$ follows the uniform distribution on $\Delta$ if and only if  $(X_1, \cdots, X_{m-1})$ follows the uniform distribution on $W$.
\end{lemma}
\begin{proof}[Proof of Lemma \ref{paint_swim}]
 First, assume that $(X_1, \cdots, X_m)$ follows the uniform distribution on $\Delta$. Then
$(X_1, \cdots, X_{m-1})^T=A(X_1, \cdots, X_m)^T$ where $A$ is the projection matrix with $A=(I_{m-1}, \mathbf{0})$ where $\mathbf{0}$ is a $(m-1)$-dimensional zero vector. Since a linear transform sends a uniform distribution to another uniform distribution [see p. 158 from Fristedt and Gray (1997)], and since  $A\Delta=W$, we get that   $(X_1, \cdots, X_{m-1})$ is uniformly distributed on $W$.

Now, assume $(X_1, \cdots, X_{m-1})$  is uniform  on $W$. First, it is well known that
\bea\lbl{silk_peer}
\mbox{the volume of}\ \Big\{(x_1, \cdots, x_{m})\in [0, 1]^{m};\, \sum_{i=1}^{m}x_i=1 \Big\}= \frac{\sqrt{m}}{(m-1)!};
\eea
see, e.g., Rabinowitz (1989). Thus, by symmetry,
\bea\lbl{moon_apple}
\mbox{the volume of } \Delta=\frac{\sqrt{m}}{m!(m-1)!}.
\eea
Therefore, to show that $(X_1, \cdots, X_m)$ has the uniform distribution on $\Delta$, it suffices to prove that, for any bounded measurable function $\varphi$ defined on $[0,1]^m$,
\bea\lbl{integral_Barnes}
E\varphi(X_1, \cdots, X_m) = \frac{m!(m-1)!}{\sqrt{m}}\int_{\Delta}\varphi(x_1, \cdots, x_m)\,dS
\eea
where the right hand side is a surface integral.
Seeing that $\mathcal{A}:\, (x_1, \cdots, x_{m-1}) \in W\to (x_1, \cdots, x_{m-1},1-\sum_{i=1}^{m-1}x_i) \in \Delta$   is a one-to-one and onto map, then by a change of variables formula [see, e.g.,  Proposition 6.6.1 from Berger and Gostiaux (1988)],
\beaa
\int_{\Delta}\varphi(x_1, \cdots, x_m)\,dS = \int_{W} \varphi\Big(x_1, \cdots, x_{m-1}, 1-\sum_{i=1}^{m-1}x_i\Big)\cdot \mbox{det}(B^TB)^{1/2}\, dx_1\cdots dx_{m-1}
\eeaa
where
\beaa
B:=\frac{\partial (x_1, \cdots, x_{m-1},1-\sum_{i=1}^{m-1}x_i)}{\partial (x_1, \cdots, x_{m-1})}=
\begin{pmatrix}
1 & 0 & \cdots & 0 \\
0 & 1 & \cdots & 0 \\
\vdots & \vdots & \vdots & \vdots\\
0 & 0 & \cdots & 1 \\
-1 & -1 & \cdots -1 & -1
\end{pmatrix}
_{m\times {(m-1)}}.
\eeaa
Trivially, $B^TB=I_{m-1} + ee^T$, where $e=(1, \cdots, 1)^T\in \mathbb{R}^{m-1}$, which has eigenvalues $1$ with $m-2$ folds and eigenvalue $m$ with one fold. Hence, $\mbox{det}(B^TB)=m$.   Thus, the right hand side of (\ref{integral_Barnes}) is identical to
\bea\lbl{light_coffee}
 m!(m-1)!\int_{W} \varphi\Big(x_1, \cdots, x_{m-1}, 1-\sum_{i=1}^{m-1}x_i\Big)\, dx_1\cdots dx_{m-1}.
\eea
It is well known that
\beaa
\mbox{the volume of}\ \Big\{(x_1, \cdots, x_{m-1})\in [0,1]^{m-1};\, \sum_{i=1}^{m-1}x_i \le 1\Big\}=\frac{1}{(m-1)!};
\eeaa
see, e.g., Stein (1966). Thus, by symmetry,
\bea\lbl{pen_sky}
\mbox{the volume of}\ W= \frac{1}{m!(m-1)!}.
\eea
This says that the density of the uniform distribution on $W$ is identical to $m!(m-1)!.$ Consequently, the left hand side of (\ref{integral_Barnes}) is equal to
\beaa
m!(m-1)!\int_{W}\varphi\Big(x_1, \cdots, x_{m-1}, 1-\sum_{i=1}^{m-1}x_i\Big)\, dx_1\cdots dx_{m-1},
\eeaa
which together with (\ref{light_coffee}) leads to (\ref{integral_Barnes}).
\end{proof}


\medskip

Fix $m\geq 2$. Let $\ml{P}_n(m)$ be the set of partitions of $n$ with lengths at most $m.$ It is known from Erd\"{o}s and Lehner (1941) that
\bea\lbl{ask_cup}
|\ml{P}_n(m)| \sim \frac{\binom{n-1}{m-1}}{m!}\sim \frac{n^{m-1}}{m!(m-1)!}
\eea
as $n\to\infty$.


Let us comment on the proof of Theorem \ref{finite_theorem} first. To show the weak convergence, for any bounded continuous function $f$ defined on $\overline{W}$, the closure of $W$, it suffices to prove
\begin{align}\label{eq:sum}
&\frac{1}{|\ml{P}_n(m)|} \sum_{(k_1, \cdots, k_m)\vdash n } f(\frac{k_1}{n}, \ldots, \frac{k_{m-1}}{n}) \to \frac{1}{\text{Vol}(W)} \int_{W} f(x_1,\ldots, x_{m-1})\,d{\bf x}
\end{align}
as $n\to \infty$. At first sight, it seems \eqref{eq:sum} can be obtained easily by using the  convergence of a multi-dimensional Riemann sum to the corresponding integral. However, the interaction among the parts $k_1,\ldots,k_{m}$ are complicated. The difficulty lies in controlling the LHS of \eqref{eq:sum} on the boundary of $\ml{P}_n(m)$ (that is, either two parts are equal or a certain part is zero), together with  the restriction $\sum_{i=1}^m k_i=n$. Therefore, we need to make extra efforts. The main proof of this section is given below.

\begin{proof}[Proof of Theorem \ref{finite_theorem}]

By Lemma \ref{Carnige}, it is enough to prove that,  under $\ml{P}_n(m)$,   $\frac{1}{n}(k_1, \cdots, k_m)$ converges weakly to  the uniform distribution on $\Delta$ as $n\to\infty$.

We first prove the case for $m=2.$ In fact, since $k_1+k_2=n$ and $k_1\geq k_2$, we have $\frac{1}{2}n\leq k_1\leq n$. Recall $W$ in (\ref{thunder_storm}). We know $W$ is the interval $(\frac{1}{2},1)$. So it is enough to check that $k_1$ has the uniform distribution on $(\frac{1}{2},1)$. Indeed, for any $x\in (\frac{1}{2}, 1),$ the distribution function of $\frac{k_1}{n}$ is given by
\beaa
P\Big((k_1, n-k_1);\, \frac{k_1}{n}\leq x\Big) & = & P\Big((k, n-k);\, \frac{n}{2}\leq k_1 \leq [nx]\Big)\\
&=&\frac{nx-\frac{1}{2}n +O(1)}{\frac{1}{2}n+O(1)}\to 2x-1
\eeaa
as $n\to \infty,$ which is exactly the cdf of the uniform distribution  on $(1/2, 1).$

As per (\ref{pen_sky}), the volume of $W$ in (\ref{thunder_storm}) equals $\frac{1}{m! (m-1)!}$. Thus the density of the uniform distribution on $W$ has the constant value of $m! (m-1)!$ on $W$. To prove the conclusion, it suffices to show the convergence of their moment generating functions, that is,
\bea\lbl{tiger_cat}
Ee^{(t_1k_1 +\cdots + t_mk_m)/n} \to Ee^{t_1\xi_{1} +\cdots + t_m\xi_{m}}
\eea
as $n\to\infty$ for all $(t_1, \cdots, t_m) \in \mathbb{R}^m$, where $(\xi_1, \cdots, \xi_{m-1})$ has the uniform distribution on $W$ by Lemma \ref{paint_swim}. We prove this by several steps.

{\it Step 1: Estimate of LHS of (\ref{tiger_cat}).} From (\ref{tiger_cat}), we know that the left hand side of (\ref{tiger_cat}) is identical to
\bea
& & \frac{1}{|\ml{P}_n(m)|} \sum_{(k_1, \cdots, k_m)} e^{(t_1k_1 +\cdots + t_mk_m)/n} \nonumber\\
& = & \frac{1}{|\ml{P}_n(m)|} \sum_{k_1> \cdots > k_m} e^{(t_1k_1 +\cdots + t_mk_m)/n} + \frac{1}{|\ml{P}_n(m)|} \sum_{k\in Q_n } e^{(t_1k_1 +\cdots + t_mk_m)/n} \lbl{drink_cloud}
\eea
where all of the sums above are taken over $\ml{P}_n(m)$ with the corresponding restrictions, and
\beaa
Q_n:=\{k=(k_1, \cdots, k_m)\vdash n;\, k_i=k_j\ \mbox{for some}\ 1\leq i<j\leq m\}.
\eeaa
Let us first estimate the size of $Q_n$. Observe
\beaa
Q_n=\cup_{i=1}^{m-1}\{k=(k_1, \cdots, k_m)\vdash n;\, k_i=k_{i+1}\}.
\eeaa
For any $\kappa=(k_1, \cdots, k_m)\vdash n$ with $k_i=k_{i+1}$, we know $k_1 + \cdots + 2k_i + k_{i+2} +\cdots +k_m=n$, which is a non-negative integer solutions of $j_1+\cdots + j_{m-1}=n$. It is easily seen that the number of non-negative integer solutions of the equation $j_1+\cdots + j_{m-1}=n$ is equal to $\binom{n+m-2}{m-2}.$ Therefore,
\bea\lbl{cow_sheep}
|Q_n| \leq (m-1)\binom{n+m-2}{m-2} \sim (m-1)\frac{n^{m-2}}{(m-2)!}
\eea
as $n\to\infty.$ Also, by (\ref{ask_cup}), $| \ml{P}_n(m) | \sim \frac{n^{m-1}}{m!(m-1)!}$. For $e^{(t_1k_1 +\cdots + t_mk_m)/n}\leq e^{|t_1| +\cdots + |t_m|}$ for all $k_i$'s, we see that the last term in (\ref{drink_cloud}) is of order $O(n^{-1})$. Furthermore, we can assume all the $k_i$'s are positive since $|\ml{P}_n(m-1)| = o(|\ml{P}_n(m)|)$. Consequently,
\bea
Ee^{(t_1k_1 +\cdots + t_mk_m)/n} & \sim & \frac{m!(m-1)!}{n^{m-1}} \sum e^{(t_1k_1 +\cdots + t_mk_m)/n} \lbl{cat_sun}
\eea
where $(k_1, \cdots, k_m)\vdash n$ in the last sum runs over all positive integers such that  $k_1>\cdots > k_m>0$.

{\it Step 2: Estimate of RHS of (\ref{tiger_cat}).} For a set $\mathcal{A}$, let $I_\mathcal{A}$ or $I(\mathcal{A})$ denote the indicator function of $\mathcal{A}$ which takes value $1$ on the set $A$ and 0 otherwise.
Review that the density function on $W$ is equal to the constant $m!(m-1)!$. For $\xi_1+\cdots + \xi_m=1$, we  have
\bea
& & Ee^{t_1\xi_{1} +\cdots + t_m\xi_{m}}\nonumber\\
&= & m! (m-1)! e^{t_m} \int_{[0,1]^{m-1}} e^{(t_1-t_m) x_1 + \cdots + (t_{m-1}-t_m) x_{m-1} } I_{\mathcal{A}} ~d x_1 \dots d x_{m-1}\nonumber\\
&= & m! (m-1)! e^{t_m}\int_{[0,1]^{m-1}} f(x_1,\cdots,x_{m-1}) I_{\mathcal{A}} ~d x_1 \dots d x_{m-1}, \lbl{hodge}
\eea
where
\bea
& &  \mathcal{A}=\Big\{(x_1, \cdots, x_{m-1})\in [0,1]^{m-1};\, x_1 > \cdots > x_{m-1} > 1- \sum_{i=1}^{m-1} x_i \ge 0\Big\}; \nonumber\\
& & f(x_1,\cdots,x_{m-1}):=e^{(t_1-t_m) x_1 + \cdots + (t_{m-1}-t_m) x_{m-1} }. \lbl{master-could}
\eea

{\it Step 3: Difference between LHS and RHS of (\ref{tiger_cat}).} Denote
\beaa
& & \mathcal{A}_n:=  \Big\{ (k_1, \cdots, k_{m-1})\in \{1,\cdots, n\}^{m-1};\, \frac{k_1}{n} > \cdots > \frac{k_{m-1}}{n} > 1- \sum_{i=1}^{m-1} \frac{k_i}{n} > 0\Big\};   \\
& & f_n (k_1,\cdots,k_{m-1}) := e^{(t_1-t_m)k_1/n +\cdots + (t_{m-1}-t_m) k_{m-1}/n}
\eeaa
for all $(k_1, \cdots, k_{m-1}) \in \mathcal{A}_n$. From (\ref{cat_sun}), we obtain
\beaa
& & Ee^{(t_1k_1 +\cdots + t_mk_m)/n} \\
&\sim & e^{t_m} \frac{m!(m-1)!}{n^{m-1}} \sum_{k_1> \cdots > k_m>0} e^{(t_1-t_m)k_1/n +\cdots + (t_{m-1}-t_m) k_{m-1}/n}\\
&= & m!(m-1)! e^{t_m} \sum_{k_1 = 1}^n \cdots \sum_{k_{m-1}=1}^n \int_{\frac{k_1-1}{n}}^{\frac{k_1}{n}} \cdots \int_{\frac{k_{m-1}-1}{n}}^{\frac{k_{m-1}}{n}} f_n (k_1,\cdots,k_m) I_{ \mathcal{A}_n } ~d x_1 \dots d x_{m-1}.
\eeaa
Writing the integral in (\ref{hodge}) similar to the above, we get that
\bea
& & Ee^{t_1\xi_{1} +\cdots + t_m\xi_{m}} - Ee^{(t_1k_1 +\cdots + t_mk_m)/n} \nonumber\\
&\sim & m! (m-1)! e^{t_m} \sum_{k_1 = 1}^n \cdots \sum_{k_{m-1}=1}^n \int_{\frac{k_1-1}{n}}^{\frac{k_1}{n}} \cdots \int_{\frac{k_{m-1}-1}{n}}^{\frac{k_{m-1}}{n}} \nonumber\\
& & \quad \quad \quad \quad \big( f (x_1,\cdots,x_{m-1}) I_{  \mathcal{A} }-f_n (k_1,\cdots,k_m) I_{ \mathcal{A}_n }  \big)~d x_1 \dots d x_{m-1}\nonumber
\eea
which again is identical to
\bea
&  & m! (m-1)! e^{t_m} \sum_{k_1 = 1}^n \cdots \sum_{k_{m-1}=1}^n \int_{\frac{k_1-1}{n}}^{\frac{k_1}{n}} \cdots \int_{\frac{k_{m-1}-1}{n}}^{\frac{k_{m-1}}{n}} \nonumber\\
& & \quad \quad \quad \quad \quad \quad  \quad \quad f (x_1,\cdots,x_{m-1}) \left( I_{  \mathcal{A} } -  I_{ \mathcal{A}_n } \right) ~d x_1 \dots d x_{m-1}\lbl{chicken_say}\\
&& \quad + m! (m-1)! e^{t_m} \sum_{k_1 = 1}^n \cdots \sum_{k_{m-1}=1}^n \int_{\frac{k_1-1}{n}}^{\frac{k_1}{n}} \cdots \int_{\frac{k_{m-1}-1}{n}}^{\frac{k_{m-1}}{n}} \nonumber\\
& &\quad \quad \quad \quad \quad \quad \quad \quad\left( f (x_1,\cdots,x_{m-1}) -  f_n (k_1,\cdots,k_{m-1})  \right) I_{ \mathcal{A}_n } ~d x_1 \dots d x_{m-1} \ \ \ \ \ \  \ \lbl{sky_leaves}\\
&= & m! (m-1)! e^{t_m} \left( \mathcal{S}_1 + \mathcal{S}_2\right),\nonumber
\eea
where $\mathcal{S}_1$ stands for the sum in (\ref{chicken_say}) and $\mathcal{S}_2$ stands for the sum in (\ref{sky_leaves}). The next step is to show both $\mathcal{S}_1 \to 0$ and $\mathcal{S}_2 \to 0$ as $n\to\infty$  and this completes the proof.

{\it Step 4: Proof of that $\mathcal{S}_2 \to 0$.} First, for the  term $\mathcal{S}_2$, given that $$\frac{k_1 - 1}{n} \le x_1 \le \frac{k_1}{n}, \cdots, \frac{k_{m-1} - 1}{n} \le x_{m-1} \le \frac{k_{m-1}}{n},$$ we have
\beaa
 | f(x_1, \cdots, x_{m-1}) - f_n(k_1, \cdots, k_{m-1}) | \leq \frac{1}{n}\exp\Big\{\sum_{i=1}^{m-1} |t_i -t_m|\Big\}\cdot \sum_{i=1}^{m-1} |t_i -t_m|.
\eeaa
Indeed, the above follows from the mean value theorem by considering $|g(1) - g(0) |$, where
$$g(s): = \exp\Big\{\sum_{i=1}^{m-1} (t_i - t_m) [ sx_i + (1-s) \frac{k_i}{n}]\Big\}.$$
Thus
$$|\mathcal{S}_2| \le  \Big(\frac{1}{n}\Big)^{m-1} n^{m-1}  \frac{\exp\big\{\sum_{i=1}^{m-1} |t_i -t_m|\big\}\cdot\sum_{i=1}^{m-1}|t_i -t_m| }{n} \to 0$$
as $n\to\infty$.

{\it Step 5. Proof of that $\mathcal{S}_1 \to 0$.} From (\ref{master-could}), we immediately see that
\bea\lbl{CCL}
\|f\|_{\infty}:=\sup_{(x_1, \cdots, x_{m-1})\in [0,1]^{m-1}}| f (x_1,\cdots,x_{m-1}) | \le e^{|t_1 - t_m| + \cdots |t_{m-1} - t_m|}.
\eea
By definition, as $k_i$ ranges from 1 to $n$ for $i=1,\dots, m-1$, the function $I_{ \mathcal{A}_n}$ equals 1 only when the followings hold
\bea\lbl{Chan}
\frac{k_1}{n} > \frac{k_2}{n}, \cdots, \frac{k_{m-2}}{n}> \frac{k_{m-1}}{n}, \frac{k_1 + \cdots k_{m-2} + 2k_{m-1}}{n} >1, \frac{k_1 + \cdots +  k_{m-1}}{n} < 1.
\eea
Similarly, $I_{\mathcal{A}}$ equals 1 only when
\bea\lbl{Tang}
x_1 > x_2, \cdots, x_{m-2}> x_{m-1}, x_1 + \cdots + x_{m-2} + 2x_{m-1} >1, x_1 + \cdots + x_{m-1} < 1.\ \ \
\eea
Let $\mathcal{B}_n$ be a subset of $\mathcal{A}_n$ such that
\beaa
\mathcal{B}_n= \mathcal{A}_n \cap \Big\{(k_1, \cdots, k_{m-1})\in \{1,2,\cdots, n\}^{m-1};\, \frac{k_{m-1}}{n} +\sum_{i=1}^{m-1} \frac{k_i}{n} > \frac{m}{n}+1\Big\}.
\eeaa
 Given $(k_1, \cdots, k_{m-1})\in \mathcal{B}_n$, for any
\bea\lbl{red_book}
\frac{k_1 - 1}{n} < x_1 < \frac{k_1}{n}, \cdots, \frac{k_{m-1} - 1}{n} < x_{m-1} < \frac{k_{m-1}}{n},
\eea
it is easy to verify from (\ref{Chan}) and (\ref{Tang}) that $I_{\mathcal{A}}=1$. Hence,
 \bea
& &  I_{\mathcal{A}_n}=I_{\mathcal{B}_n} + I_{\mathcal{A}_n\backslash\mathcal{B}_n} \nonumber\\
& \leq & I_{\mathcal{A}} + I\Big\{(k_1, \cdots, k_{m-1})\in \{1,\cdots, n\}^{m-1};\, 1 < \frac{k_{m-1}}{n} + \sum_{i=1}^{m-1} \frac{k_i}{n} \leq  \frac{m}{n}+1\Big\} \nonumber\\
& = & I_{\mathcal{A}} + \sum_{j=n+1}^{n+m}I_{E_j}\lbl{spring_cold}
 \eea
 where
\beaa
E_j:=\Big\{(k_1, \cdots, k_{m-1})\in \{1,\cdots, n\}^{m-1};\, k_1+\cdots + k_{m-2} +2k_{m-1}=j\Big\}
\eeaa
for $n+1 \leq j \leq m+n$. Similar to the argument as in {\it Step 1},
\bea\lbl{us_pen}
\max_{n\leq j \leq m+n}|E_j| =O(n^{m-2})
\eea
as $n\to\infty$. On the other hand, consider a subset of $\mathcal{A}_n^c:=\{1,\cdots, n\}^{m-1}\backslash \mathcal{A}_n$ defined by
\beaa
\mathcal{C}_n &:=&
\Big\{(k_1, \cdots, k_{m-1})\in \{1,2,\cdots, n\}^{m-1};\, \mbox{either}\  k_i\leq k_{i+1}-1\ \mbox{for some }
  1\leq i \leq m-2,\\
& & \ \mbox{or}\
 k_1+\cdots + k_{m-2} + 2k_{m-1} \leq   n,\ \mbox{or}\ k_1+\cdots  + k_{m-1} \geq  m+n-1 \Big\}.
\eeaa
Set $\mathcal{A}^c=[0,1]^{m-1}\backslash \mathcal{A}$. Given $(k_1, \cdots, k_{m-1})\in \mathcal{C}_n$, for any $k_i$'s and $x_i$'s satisfying (\ref{red_book}), it is not difficult to check that $I_{\mathcal{A}^c}=1$. Consequently,
\beaa
I_{\mathcal{A}_n^c} &= & I_{\mathcal{C}_n} + I\Big\{(k_1, \cdots, k_{m-1})\in \mathcal{A}_n^c;\, k_i> k_{i+1}-1\ \mbox{for all }\ 1\leq i \leq m-2,\\
& & ~~~~~~~~~~~~ k_1+\cdots + k_{m-2} + 2k_{m-1} >  n,\ \mbox{and}\ k_1+\cdots  + k_{m-1} <  m+n-1 \Big\}\\
& \leq & I_{\mathcal{A}^c} + I(\mathcal{D}_{n,1}) + I(\mathcal{D}_{n,2}),
\eeaa
or equivalently,
\bea\lbl{coca_warm}
I_{\mathcal{A}_n}
 \geq  I_{\mathcal{A}} - I(\mathcal{D}_{n,1}) -  I(\mathcal{D}_{n,2}),
\eea
where
\beaa
& & \mathcal{D}_{n,1}:=\bigcup_{i=1}^{m-2}\big\{(k_1, \cdots, k_{m-1})\in \{1,2,\cdots, n\}^{m-1};\, k_i=k_{i+1}\big\};\\
& & \mathcal{D}_{n,2}:=\bigcup_{i=n}^{ n+m-2}\big\{(k_1, \cdots, k_{m-1})\in \{1,2,\cdots, n\}^{m-1};\, k_1+\cdots +k_{m-1}=i\big\}.
\eeaa
By the same argument as in (\ref{cow_sheep}), we have $\max_{1\leq i \leq 2}|\mathcal{D}_{n,i}|=O(n^{m-2})$ as $n\to\infty$. Joining (\ref{spring_cold}) and (\ref{coca_warm}), and assuming (\ref{red_book}) holds, we arrive at
\beaa
|I_{\mathcal{A}_n}-  I_{\mathcal{A}}|\leq  I(\mathcal{D}_{n,1})+ I(\mathcal{D}_{n,2})  + \sum_{i=n+1}^{n+m}I_{E_i}
\eeaa
and $\sum_{i=1}^2|\mathcal{D}_{n,i}| + \sum_{i=n+1}^{n+m}|E_i|=O(n^{m-2})$ as $n\to\infty$ by (\ref{us_pen}). Review $\mathcal{S}_1$  in (\ref{chicken_say}). Observe that $\mathcal{D}_{n,i}$'s and $E_i$'s do not depend on $x$, we obtain from (\ref{CCL}) that
\beaa
 \mathcal{S}_1 &\leq & \|f\|_{\infty}\cdot\sum_{k_1 = 1}^n \cdots \sum_{k_{m-1}=1}^n \Big[\sum_{i=1}^2I(\mathcal{D}_{n,i}) + \sum_{i=n}^{n+m}I_{E_i}\Big]\int_{\frac{k_1-1}{n}}^{\frac{k_1}{n}} \cdots \int_{\frac{k_{m-1}-1}{n}}^{\frac{k_{m-1}}{n}}1 ~d x_1 \dots d x_{m-1}\\
& = &  \|f\|_{\infty}\cdot \Big(\sum_{i=1}^2|\mathcal{D}_{n,i}| + \sum_{i=n}^{n+m}|E_i|\Big)\cdot \frac{1}{n^{m-1}}\\
& = & O(n^{-1})
\eeaa
as $n\to\infty.$ The proof is completed.

\end{proof}

\subsection{Proof of Theorem \ref{cancel_temple}}\label{sec:proof:restricted-uniform}


We first rewrite the eigenvalues of the Laplace-Beltrami operator given in (\ref{Jingle}) in terms of $\kappa$  instead of a mixing of $\kappa$ and $\kappa'$. A  similar expression, which is essentially the same as ours, can be found on p. 596 from Dumitriu {\it et al}. (2007). So we skip the proof.

\begin{lemma}\lbl{theater} Let $\alpha>0$. Let $\lambda_{\kappa}$ be as in (\ref{Jingle}). For $\kappa=(k_1, \cdots, k_m) \vdash n$,  we have
\bea\lbl{green}
\lambda_{\kappa}=\big(m-\frac{\alpha}{2}\big)n +\sum_{i=1}^m (\frac{\alpha}{2}k_i-i)k_i.
\eea
\end{lemma}




 Let $\eta$ follow the chi-square distribution $\chi^2(v)$ with density function
\bea\lbl{his_her}
(2^{v/2}\Gamma(v/2))^{-1}x^{\frac{v}{2}-1}e^{-x/2}, ~~~~~~~ x>0.
\eea
The following lemma is on p. 486 from Kotz {\it et al}. (2000).
\begin{lemma}\lbl{Kotz} Let $m\geq 2$ and $\eta_1, \cdots, \eta_m$ be independent random variables with $\eta_i \sim \chi^2(v_i)$ for each $i$. Set $X_i=\eta_{i}/(\eta_1+\cdots + \eta_m)$ for each $i.$ Then $(X_1, \cdots, X_{m-1})$ has density
\beaa
f(x_1, \cdots, x_{m-1})=\frac{\Gamma(\frac12\sum_{j=1}^{m}v_j)}{\prod_{j=1}^{m}\Gamma(\frac12v_j)}\Big[\prod_{j=1}^{m-1}x_j^{(v_j/2)-1}\Big]
\Big(1-\sum_{j=1}^{m-1}x_j\Big)^{(v_m/2)-1}
\eeaa
on the set $U=\{(x_1, \cdots, x_{m-1})\in [0,1]^{m-1};\, \sum_{i=1}^{m-1}x_i\leq 1\}$.
\end{lemma}

\begin{proof}[Proof of Theorem \ref{cancel_temple}]
By Lemma \ref{theater}, for $m$ is fixed and $k_1\leq n$, we have
\beaa
\frac{\lambda_{\kappa}}{n^2}=\frac{\alpha}{2}\cdot\sum_{i=1}^m\Big(\frac{k_i}{n}\Big)^2 + o(1)
\eeaa
as $n\to\infty.$ By Theorem \ref{finite_theorem}, under the uniform distribution on either  $\ml{P}_n(m)$ or $\ml{P}_n(m)'$, $\frac{1}{n}(k_1, \cdots, k_m)$ converges weakly to $(Z_1, \cdots, Z_m)$, which has the uniform measure on $\Delta$. Note that $\Delta$ is the ordered simplex, hence we can not get the desired conclusion by directly applying (i) or (ii) from the Comments after the statement of Theorem \ref{cancel_temple}. We will resolve this issue next.
%

Let $\xi_1, \cdots, \xi_m$ be independent random variables with the common density $e^{-x}I(x\geq 0)$. Set
\beaa
S_m=\xi_1 + \cdots + \xi_m~~~ \mbox{and} ~~~ X_i=\frac{\xi_{(i)}}{S_m}, \ \ \ 1\leq i \leq m
\eeaa
where $\xi_{(1)}>\cdots>\xi_{(m)}$ are the order statistics.  By the continuous mapping theorem and the fact $\sum_{i=1}^m\xi_{(i)}^2=\sum_{i=1}^m\xi_{i}^2$, we only need to show that $(Z_1, \cdots, Z_m)$ has the same distribution as that of $(X_1, \cdots, X_m)$. Review $W$ in Lemma \ref{paint_swim}. Recall that the volume of the convex body $W$ (as per \eqref{pen_sky}) is $(m!(m-1)!)^{-1}$. Therefore, by Lemma \ref{paint_swim}, it suffices to prove that
\bea\lbl{Big_sing}
E\varphi(X_1, \cdots, X_{m-1})=m!(m-1)!\int_{W}\varphi(x_1, \cdots, x_{m-1})\,dx_1\cdots dx_{m-1}
\eea
for any bounded and measurable function $\varphi$ defined on $[0, 1]^{m-1}.$
Recalling (\ref{his_her}), we know $\chi^2(2)/2$ has the exponential density function  $e^{-x}I(x\geq 0)$.
Taking $v_1=v_2=\cdots =v_m=2$ in Lemma \ref{Kotz}, we see that the density function of $\big(\frac{\xi_{1}}{S_m}, \cdots, \frac{\xi_{m-1}}{S_m}\big)$
on $U$ is equal to the constant $\Gamma(m)=(m-1)!$. Furthermore,
\beaa
E\varphi(X_1, \cdots, X_{m-1})
=  \sum_{\pi}E\Big[\varphi\Big(\frac{\xi_{\pi(1)}}{S_m}, \cdots, \frac{\xi_{\pi(m-1)}}{S_m}\Big)I(\xi_{\pi(1)}>\cdots > \xi_{\pi(m)})\Big],
\eeaa
where the sum  is taken over every permutation $\pi$ of $m.$ Write $S_m=\xi_{\pi(1)}+\cdots + \xi_{\pi(m)}.$ By the i.i.d. property of $\xi_i$'s, we get
\beaa
& & E\varphi(X_1, \cdots, X_{m-1})\\
& = &  m!\cdot E\Big[\varphi\Big(\frac{\xi_{(1)}}{S_m}, \cdots, \frac{\xi_{(m-1)}}{S_m}\Big)
I\Big(\frac{\xi_{(1)}}{S_m}> \cdots > \frac{\xi_{(m-1)}}{S_m}> 1-\frac{\sum_{i=1}^{m-1}\xi_{(i)}}{S_m}\Big)\Big]\\
& =& m!(m-1)!\int_{U}\varphi(x_1, \cdots, x_{m-1})I\Big(x_1> \cdots > x_{m-1}> 1-\sum_{i=1}^{m-1}x_i\Big)\,dx_1\cdots dx_{m-1}
\eeaa
for $\big(\frac{\xi_{(1)}}{S_m}, \cdots, \frac{\xi_{(m-1)}}{S_m}, 1-\frac{\sum_{i=1}^{m-1}\xi_{(i)}}{S_m}\big)$ is a function of $\big(\frac{\xi_{1}}{S_m}, \cdots, \frac{\xi_{m-1}}{S_m}\big)$ which has a constant density $(m-1)!$ on $U$ as shown earlier. Easily, the last term above is equal to the right hand side of (\ref{Big_sing}). The proof is then completed.
\end{proof}

\subsection{Proof of Theorem \ref{Gamma_surprise}}\label{sec:proof:restricted-Jack}

We start with a result on the restricted Jack probability measure $P_{n,m}^\alpha$ as in (\ref{eq:restrictedjack}).

\begin{lemma}\lbl{Sho_friend}(Matsumoto, 2008). Let $\alpha>0$ and $\beta=2/\alpha.$ For a given integer $m\geq 2$, let $\kappa=(k_{n,1}, \cdots, k_{n,m})\vdash n$ be chosen with probability  $P_{n,m}^{ \alpha}(\kappa).$ Then, as $n\to\infty$,
\beaa
\Big(\sqrt{\frac{\alpha m}{n}}\big(k_{n,i}-\frac{n}{m}\big)\Big)_{1\leq i \leq m}
\eeaa
converges weakly to a limiting distribution with density function
\bea\lbl{dance}
g(x_1, \cdots, x_m)=\mbox{const}\cdot e^{-\frac{\beta}{2}\sum_{i=1}^mx_i^2}\cdot \prod_{1\leq j < k \leq m}|x_j - x_k|^{\beta}
\eea
for all $x_1\geq x_2\geq \cdots \geq x_m$ such that $x_1+\cdots + x_m=0.$
\end{lemma}

The idea of the proof of Theorem \ref{Gamma_surprise} below lies in that, by virtue of Lemma \ref{Sho_friend}, we are able to write $\lambda_{\kappa}$ in (\ref{Jingle}) in terms of the trace of a Wishart matrix. Due to this we get the Gamma density by evaluating the moment generating function (or the Laplace transform) of the trace through (\ref{dance}).
\begin{proof}[Proof of Theorem \ref{Gamma_surprise}]
Let
\beaa
Y_{n,i}=\sqrt{\frac{\alpha m}{n}}\big(k_{n,i}-\frac{n}{m}\big)
\eeaa
for $1\leq i \leq m$. By Lemma \ref{Sho_friend}, under $P_{n,m}^{\alpha}$,  we know $(Y_{n,1}, \cdots, Y_{n,m})$ converges weakly to a random vector $(X_1, \cdots, X_m)$ with density function $g(x_1, \cdots, x_m)$ as in (\ref{dance}). Checking the proof of Lemma \ref{Sho_friend}, it is easy to see that its conclusion still holds for $Q_{n,m}^{\alpha}$ without changing its proof. Solve for $k_{n,i}$'s to have
\beaa
k_{n,i}=\frac{n}{m} + \sqrt{\frac{n}{\alpha m}}Y_{n,i}
\eeaa
for $1\leq i \leq m$. Substitute these for the corresponding terms in \eqref{green} to see that
\beaa
& & \lambda_\kappa-\big(m-\frac{\alpha}{2}\big)n\\
& = & \sum_{i=1}^m\Big[\frac{\alpha}{2}\big(\frac{n}{m} +\sqrt{\frac{n}{m\alpha}}Y_{n,i}\big)-i\Big]\cdot \big(\frac{n}{m} +\sqrt{\frac{n}{m\alpha}}Y_{n,i}\big)\\
& = & \frac{\alpha}{2}\sum_{i=1}^m\big(\frac{n}{m} +\sqrt{\frac{n}{m\alpha}}Y_{n,i}\big)^2-\sum_{i=1}^m i\big(\frac{n}{m} +\sqrt{\frac{n}{m\alpha}}Y_{n,i}\big)\\
& = & \frac{\alpha}{2}\cdot\frac{n^2}{m} + \sqrt{\alpha}\cdot \big(\frac{n}{m}\big)^{3/2}\sum_{i=1}^mY_{n,i}+\frac{n}{2m}\sum_{i=1}^mY_{n,i}^2
- \frac{n(m+1)}{2}- \sqrt{\frac{n}{m\alpha}} \sum_{i=1}^m i Y_{n,i}\\
& = & \frac{\alpha}{2}\cdot\frac{n^2}{m} - \frac{n(m+1)}{2}+\frac{n}{2m}\sum_{i=1}^mY_{n,i}^2
- \sqrt{\frac{n}{m\alpha}} \sum_{i=1}^m i Y_{n,i}
\eeaa
since $\sum_{i=1}^mY_{n,i}=0.$ According to the notation of $a_n$ and $b_n$,
\beaa
\frac{\lambda_\kappa-a_n}{b_n}=\sum_{i=1}^mY_{n,i}^2
- \frac{2}{\sqrt{\alpha}}\sqrt{\frac{m}{n}} \sum_{i=1}^m i Y_{n,i}.
\eeaa
Since $(Y_{n,1}, \cdots, Y_{n,m})$ converges weakly to the random vector $(X_1, \cdots, X_m)$, taking $$h_1(y_1, \cdots, y_m)=\sum_{i=1}^miy_i \quad \text{and} \quad h_2(y_1, \cdots, y_m)=\sum_{i=1}^my_i^2,$$ respectively, by the continuous mapping theorem,
\beaa
\sum_{i=1}^miY_{n,i}\to \sum_{i=1}^miX_i\ \ \ \mbox{and}\ \ \ \ \sum_{i=1}^mY_{n,i}^2\to \sum_{i=1}^mX_i^2
\eeaa
weakly as $n\to\infty$. By the Slutsky lemma,
\beaa
\frac{\lambda_\kappa-a_n}{b_n}=\sum_{i=1}^mY_{n,i}^2 + O_p\big(n^{-1/2}\big) \to \sum_{i=1}^mX_i^2
\eeaa
weakly as $n\to\infty.$ Now let us calculate the moment generating function of $\sum_{i=1}^mX_i^2$. Recall (\ref{dance}). Let $C_n$ be the normalizing constant such that
\beaa
g(x_1, \cdots, x_m)=C_m\cdot e^{-\frac{\beta}{2}\sum_{i=1}^mx_i^2}\cdot \prod_{1\leq j < k \leq m}|x_j - x_k|^{\beta}
\eeaa
is a probability density function on the subset of $\mathbb{R}^{m}$ such that  $x_1\geq x_2\geq \cdots \geq x_m$ and  $x_1+\cdots + x_m=0.$
 We then have
\bea
Ee^{t\sum_{i=1}^mX_i^2}&=& C_m\int_{\mathbb{R}^{m-1}}e^{t\sum_{i=1}^mx_i^2}g(x_1, \cdots, x_m)\,dx_1, \cdots, dx_{m-1} \nonumber\\
& = & C_m \int_{\mathbb{R}^{m-1}}e^{-\frac{\beta}{2}\sum_{i=1}^m(1-\frac{2t}{\beta})x_i^2}\prod_{1\leq j < k \leq m}|x_j - x_k|^{\beta}\,dx_1, \cdots, dx_{m-1}\nonumber\\
& = & \Big(1-\frac{2t}{\beta}\Big)^{-\frac{1}{2}\cdot(\frac{m(m-1)}{2}\beta + (m-1))}\cdot \int_{\mathbb{R}^{m-1}}g(y_1, \cdots, y_m)\,dy_1, \cdots, dy_{m-1}\nonumber\\
& = & \Big(1-\frac{2t}{\beta}\Big)^{-\frac{1}{4}(m-1)\cdot(m\beta + 2)} \lbl{laugh_sniff}
\eea
for $t<\frac{\beta}{2}$, where a transform $y_i=(1-\frac{2t}{\beta})^{1/2}x_i$ is taken in the third step for $i=1,\cdots, m-1.$ It is easy to check that the term in (\ref{laugh_sniff}) is also the generating function of the Gamma distribution with density function $h(x)=\frac{1}{\Gamma(v)\, (2/\beta)^{v}}x^{v-1}e^{-\beta x/2}$
for all $x\geq 0$, where $v=\frac{1}{4}(m-1)\cdot(m\beta + 2).$ By the uniqueness theorem, we know the conclusion holds.
\end{proof}

\subsection{Proof of Theorem \ref{vase_flower}}\lbl{proof_vase_flower}

The following lemma is Theorem 2 from Pittel (1997).

\begin{lemma}\lbl{wuliangye}
Let $\kappa=(k_1, \cdots, k_m)$ be a partition of $n$ chosen according to the uniform measure on $\mathcal{P}(n)$. Then
\beaa
k_j=
\begin{cases}
\big(1+ O_p((\log n)^{-1})\big) E(j)\ \ \ \ \ \text{if\ \ $1\leq j \leq \log n$;}\\
E(j) + O_p(nj^{-1}\log n)^{1/2}\ \ \ \ \ \text{if\ \ $\log n\leq j \leq n^{1/2}$;} \\
E(j) + O_p(e^{-cjn^{-1/2}}n^{1/2}\log n)^{1/2} \ \ \ \ \ \text{if\ \ $n^{1/2}\leq j \leq \kappa_n$};\\
(1+O_p(a_n^{-1}))E(j)\ \ \ \ \ \text{if\ \ $\kappa_n \leq j \leq k_n$}
\end{cases}
\eeaa
uniformly as $n\to\infty$, where $c=\pi/\sqrt{6}$,
\beaa
& & E(x)=\frac{\sqrt{n}}{c}\log \frac{1}{1-e^{-cxn^{-1/2}}} \ \ \ \mbox{for $x>0$}, \\
& &   \kappa_n=\left[\frac{\sqrt{n}}{4c}\log n\right]
\ \ \ \mbox{and}\ \ \ \ k_n=\Big[\frac{\sqrt{n}}{2c}(\log n-2\log\log n-a_n)\Big]
\eeaa
with $a_n\to\infty$ and $a_n=o(\log \log n)$ as $n\to\infty.$
\end{lemma}

Based on Lemma \ref{wuliangye}, we get the following law of large numbers. This is a key estimate in the proof of Theorem \ref{vase_flower}.
\begin{lemma}\lbl{happy_uniform}
Let $\kappa=(k_1, \cdots, k_m)$ be a partition of $n$ chosen according to the uniform measure on $\mathcal{P}(n)$. Then $n^{-3/2}\sum_{j=1}^m k_j^2 \to a$ in probability as $n\to\infty$, where
\bea\lbl{learn}
a=c^{-3}\int_0^1\frac{\log^2 (1-t)}{t}\,dt
\eea
and $c=\pi/\sqrt{6}$. The above conclusion also holds if ``$\sum_{j=1}^m k_j^2$" is replaced by ``$2\sum_{j=1}^m jk_j$".
\end{lemma}
\begin{proof}[Proof of Lemma \ref{happy_uniform}]
Define
\beaa
F(x)=\log \frac{1}{1-e^{-cxn^{-1/2}}}
\eeaa
for $x>0.$ Note that both $E(x)$ and $F(x)$ are decreasing in $x\in (0, \infty).$

{\it Step 1}. We first claim that
\bea\lbl{rubber_band}
\max_{1\leq j \leq \frac{1}{6}\sqrt{n}\log n}\big|\frac{k_j}{E(j)}-1\big| \to 0
\eea
in probability as $n\to\infty.$ (The choice of $1/6$ is rather arbitrary here. Actually, any number strictly less than $1/2c$ would work). We prove this next.

Notice
\beaa
\max_{x\geq 1}E(x)=E(1)
&=& -\frac{\sqrt{n}}{c} \log \big(1-e^{-cn^{-1/2}}\big)\nonumber\\
& \sim & -\frac{\sqrt{n}}{c} \log \big(cn^{-1/2}\big) \sim \frac{1}{2c}\sqrt{n}\log n
\eeaa
as $n\to\infty$ since $1-e^{-x}\sim x$ as $x\to 0.$ Observe
\beaa
\frac{\sqrt{nj^{-1}\log n}}{E(j)}=-c\sqrt{\log n}\cdot \frac{j^{-1/2}}{\log \big(1-e^{-cjn^{-1/2}}\big)}.
\eeaa
Therefore,
\beaa
\max_{\log n \leq j \leq (\log n)^2}\frac{\sqrt{nj^{-1}\log n}}{E(j)}\leq \frac{c}{F(\log^2n)} \to 0
\eeaa
and
\beaa
\max_{\log^2n \leq j \leq n^{1/2}}\frac{\sqrt{nj^{-1}\log n}}{E(j)} \leq c\frac{(\log n)^{-1/2}}{F(n^{1/2})} \to 0
\eeaa
as $n\to\infty$. By Lemma \ref{wuliangye},
\bea\lbl{mine}
\max_{\log n \leq j \leq \sqrt{n}}\Big|\frac{k_j}{E(j)}-1\Big| =o_p(1)
\eea
as $n\to\infty$. Now we consider the case for $n^{1/2}\leq j \leq \kappa_n$ where $\kappa_n$ is as in Lemma \ref{wuliangye}. Trivially, $\frac{1}{4c} >\frac{1}{6}$. Notice that
\beaa
\max_{n^{1/2}\leq j \leq (1/6)\sqrt{n}\log n}\frac{(e^{-cjn^{-1/2}}n^{1/2}\log n)^{1/2}}{E(j)}
&\leq & \frac{(e^{-c}n^{1/2}\log n)^{1/2}}{E((1/6)\sqrt{n}\log n)} \\
&= &\frac{(ce^{-c/2})n^{-1/4}(\log n)^{1/2}}{F((1/6)\sqrt{n}\log n)}.
\eeaa
Evidently,
\bea\lbl{child_baby}
F\big(\frac{1}{6}\sqrt{n}\log n\big)=-\log \big(1-e^{-(c/6)\log n}\big) \sim \frac{1}{n^{c/6}}
\eea
as $n\to\infty.$ This says
\beaa
\max_{n^{1/2}\leq j \leq (1/6)\sqrt{n}\log n}\Big|\frac{k_j}{E(j)}-1\Big| =o_p(1)
\eeaa
as $n\to\infty$ by Lemma \ref{wuliangye}. This together with (\ref{mine}) and the first expression of $k_j$ in Lemma \ref{wuliangye}  concludes  (\ref{rubber_band}), which is equivalent to that
\bea
\lbl{skill}
k_j=E(j) +\epsilon_{n,j}E(j)
\eea
uniformly for all $1\leq j\leq  (1/6)\sqrt{n}\log n$, where $\epsilon_{n,j}$'s satisfy
\bea\lbl{gymnastics}
H_n:=\sup_{1\leq j\leq (1/6)\sqrt{n}\log n}|\epsilon_{n,j}| \to 0
\eea
in probability as $n\to\infty$.

{\it Step 2}. We approximate the two sums in (\ref{sing}) and (\ref{song}) below by integrals in this step. The assertions (\ref{skill}) and (\ref{gymnastics}) imply that
\bea
& & \sum_{1\leq j \leq (1/6)\sqrt{n}\log n }k_j^2 =\Big(\sum_{1\leq j \leq (1/6)\sqrt{n}\log n}E(j)^2\Big) (1+o_p(1)); \lbl{sing}\\
& & \sum_{1\leq j \leq (1/6)\sqrt{n}\log n }jk_j =\Big(\sum_{1\leq j \leq (1/6)\sqrt{n}\log n}jE(j)\Big) (1+o_p(1))\lbl{song}
\eea
as $n\to\infty.$ For $E(x)$ is decreasing in $x$ we have
\beaa
\int_1^mE(x)^2\,dx=\sum_{j=1}^{m-1}\int_j^{j+1}E(x)^2\,dx \leq \sum_{j=1}^{m-1}E(j)^2
\eeaa
for any $m\geq 2.$ Consequently,
\beaa
\sum_{1\leq j \leq (1/6)\sqrt{n}\log n}E(j)^2 \geq \int_1^{m_1}E(x)^2\,dx
\eeaa
with $m_1=\big[\frac{1}{6}\sqrt{n}\log n\big]$. Similarly,
\beaa
\int_0^{m+1}E(x)^2\,dx=\sum_{j=0}^{m}\int_j^{j+1}E(x)^2\,dx \geq \sum_{j=1}^{m+1}E(j)^2
\eeaa
for any $m\geq 1.$ The two inequalities imply
\bea\lbl{sneeze}
\int_1^{m_1}E(x)^2\,dx \leq \sum_{1\leq j \leq (1/6)\sqrt{n}\log n}E(j)^2 \leq \int_0^{\infty}E(x)^2\,dx.
\eea
By the same argument,
\bea\lbl{rain_snow}
\int_1^{m_1}E(x)\,dx \leq \sum_{1\leq j \leq (1/6)\sqrt{n}\log n}E(j) \leq \int_0^{\infty}E(x)\,dx.
\eea
Now we estimate $\sum_{1\leq j \leq (1/6)\sqrt{n}\log n}j E(j)$. Use the inequality
\beaa
jE(j+1)\leq \int_j^{j+1}xE(x)\,dx \leq (j+1) E(j)
\eeaa
to have
\beaa
(j+1)E(j+1)-E(j+1)\leq \int_j^{j+1}xE(x)\,dx \leq j E(j) + E(j)
\eeaa
for all $j\geq 0$. Sum the inequalities over $j$ and use (\ref{rain_snow}) to get
\bea
\int_1^{m_1}x E(x)\,dx - \int_0^{\infty}E(x)\,dx &\leq & \sum_{1\leq j \leq (1/6)\sqrt{n}\log n}jE(j) \nonumber\\
& \leq & \int_0^{\infty}xE(x)\,dx + \int_0^{\infty}E(x)\,dx.\lbl{did}
\eea

{\it Step 3}. In this step, we evaluate integrals $\int E(x)\,dx$, $\int E(x)^2\,dx$ and $\int xE(x)\,dx$. First,
\beaa
\int_0^{\infty} E(x)\,dx= \frac{\sqrt{n}}{c}\int_0^{\infty}\log \frac{1}{1-e^{-cxn^{-1/2}}}\,dx.
\eeaa
Set
\bea\lbl{same_city}
t=e^{-cxn^{-1/2}}\ \ \mbox{then}\ \ x=\frac{\sqrt{n}}{c}\log \frac{1}{t}\ \ \mbox{and}\ \ dx=-\frac{\sqrt{n}}{ct}dt.
\eea
Hence
\bea\lbl{scissor}
\int_0^{\infty} E(x)\,dx=\frac{n}{c^2}\int_0^1\frac{\log (1-t)}{-t}\,dt=O(n)
\eea
as $n\to \infty$ considering the second integral above is finite. Using the same discussion, we have
\beaa
& & \int_0^{\infty} E(x)^2\,dx=\frac{n^{3/2}}{c^3}\int_0^1\frac{\log^2 (1-t)}{t}\,dt; \\
& & \int_0^{\infty} xE(x)\,dx=\frac{n^{3/2}}{c^3}\int_0^1\frac{1}{t}\log \frac{1}{t}\log\frac{1}{1-t}\,dt.
\eeaa
By the two identities above (3.44) from Pittel (1997), we have
\bea\lbl{little_brother}
\int_0^1\frac{\log^2 (1-t)}{t}\,dt=2\int_0^1\frac{1}{t}\log \frac{1}{t}\log\frac{1}{1-t}\,dt.
\eea
From the same calculation as in (\ref{same_city}), we see that
\beaa
\int_1^{m_1}E(x)^2\,dx=\frac{n^{3/2}}{c^3}\int_{e^{-cm_1n^{-1/2}}}^{e^{-cn^{-1/2}}}\frac{\log^2 (1-t)}{t}\,dt \sim \frac{n^{3/2}}{c^3}\int_0^1\frac{\log^2 (1-t)}{t}\,dt
\eeaa
as $n\to\infty$ since $m_1=\big[\frac{1}{6}\sqrt{n}\log n\big]$. By the same reasoning,
\beaa
\int_1^{m_1}xE(x)\,dx \sim \frac{n^{3/2}}{2c^3}\int_0^1\frac{\log^2 (1-t)}{t}\,dt.
\eeaa
The above two integrals and that in (\ref{scissor}) join (\ref{sneeze}), (\ref{rain_snow}) and (\ref{did}) to conclude
\bea
& & \sum_{1\leq j \leq (1/6)\sqrt{n}\log n}E(j)^2 \sim \frac{n^{3/2}}{c^3}\int_0^1\frac{\log^2 (1-t)}{t}\,dt; \lbl{music_hi}\\
& &  \sum_{1\leq j \leq (1/6)\sqrt{n}\log n}j E(j) \sim \frac{n^{3/2}}{2c^3}\int_0^1\frac{\log^2 (1-t)}{t}\,dt \lbl{sun_hello}
\eea
as $n\to\infty.$

{\it Step 4}. We will get the desired conclusion in this step. Now connecting (\ref{music_hi}) and (\ref{sun_hello}) with (\ref{sing}) and (\ref{song}) we obtain
\bea
& & \sum_{1\leq j \leq (1/6)\sqrt{n}\log n }k_j^2 = a n^{3/2} (1+o_p(1)); \lbl{sing2}\\
& & \sum_{1\leq j \leq (1/6)\sqrt{n}\log n }jk_j =\frac{a}{2} n^{3/2}(1+o_p(1)) \lbl{green_carpet}
\eea
as $n\to\infty$, where ``$a$" is as in (\ref{learn}).  For the number of parts of $\kappa=(k_1,\cdots,k_m)$, Erd\"os and Lehner (1941) obtain that
\bea\lbl{seabed}
\frac{\pi}{ \sqrt{6n} } m - \log \frac{ \sqrt{6n} }{\pi} \to \mu
\eea
weakly as $n\to\infty$ where $\mu$ is a probability measure with cdf $F_{\mu}(v) = e^{-e^{-v}}$
 for every $v \in \mathbb{R}$. See also Fristedt (1993). This implies that
\bea\lbl{red_oak}
P\big(m > \frac{1}{c} \sqrt{n}\log n \big) \to 0
\eea
as $n\to\infty$. Now, for any $\epsilon>0$, by (\ref{sing2}),
\bea
& & P\Big(\big|a-n^{-3/2}\sum_{j=1}^m k_j^2\big|\geq \epsilon\Big) \nonumber\\
& \leq & P\Big(\big|a-n^{-3/2}\sum_{1\leq j \leq \frac{1}{6}\sqrt{n}\log n} k_j^2\big|\geq \epsilon/2\Big) + P\Big( n^{-3/2} \sum_{ \frac{1}{6}\sqrt{n}\log n \leq j \leq m }k_j^2 \geq \epsilon/2 \Big) \nonumber\\
& \leq &  P\Big(m > \frac{1}{c} \sqrt{n}\log n \Big) +  P \Big(n^{-3/2} \sum_{ \frac{1}{6}\sqrt{n}\log n \leq j \leq m }k_j^2 \geq \epsilon/2,  m \leq \frac{1}{c} \sqrt{n}\log n \Big) +o(1) \nonumber\\
& \leq &  P \Big(n^{-3/2} \sum_{ \frac{1}{6}\sqrt{n}\log n \leq j \leq \frac{1}{c} \sqrt{n}\log n }k_j^2 \geq \epsilon/2 \Big) +o(1) \lbl{money_monkey}
\eea
as $n\to\infty$.
Denote by $l_n$  the least integer greater than or equal to $\frac{1}{6}\sqrt{n}\log n$. Seeing that $k_j$ is  decreasing in $j$, it is seen from (\ref{skill}) and then (\ref{child_baby}) that
\bea
k_j  \le  k_{l_n}&=& E(l_n)(1+o_p(1)) \nonumber\\
& \leq &  E\big(\frac{1}{6}\sqrt{n}\log n\big)(1+ o_p(1)) \nonumber\\
& \sim & c^{-1} n^{(1/2)-(c/6)} (1+ o_p(1)) \lbl{donkey_elephant}
\eea
for all $\frac{1}{6}\sqrt{n}\log n \leq j \leq \frac{1}{c} \sqrt{n}\log n$ as $n\to\infty$. This implies
\beaa
n^{-3/2} \sum_{ \frac{1}{6}\sqrt{n}\log n \leq j \leq \frac{1}{c} \sqrt{n}\log n }k_j^2
& \le & C\cdot n^{-3/2}  \sqrt{n}(\log n)  \big(n^{1/2-c/6}\big)^2 (1+o_p(1)) \\
& \sim & Cn^{-c/3} (\log n) (1+ o_p(1))= o_p(1)
\eeaa
as $n\to\infty$, where $C$ is a constant. This together with (\ref{money_monkey}) yields the first conclusion of the lemma. Similarly, by (\ref{green_carpet}) and (\ref{red_oak}), for any $\epsilon>0$,
\beaa
& &  P\Big(\big|\frac{a}{2}-n^{-3/2}\sum_{j=1}^m j k_j\big| \geq  \epsilon\Big)\\
 &\leq &  P\Big(\big| \frac{a}{2} -n^{-3/2}\sum_{1\leq j \leq \frac{1}{6}\sqrt{n}\log n} j k_j \big|\geq \epsilon/2 \Big)\\
 & & ~~~ + P \Big(n^{-3/2} \sum_{ \frac{1}{6}\sqrt{n}\log n \leq j \leq \frac{1}{c} \sqrt{n}\log n } j k_j \geq \epsilon/2 \Big)+  P\Big(m > \frac{1}{c} \sqrt{n}\log n \Big) \to 0
\eeaa
as $n\to\infty$ considering
\beaa
n^{-3/2}\sum_{ \frac{1}{6}\sqrt{n}\log n \leq j \leq \frac{1}{c}\sqrt{n}\log n }jk_j&\leq & C\cdot n^{-3/2}\cdot n^{(1/2)-(c/6)}(\sqrt{n}\log n)^2(1+o_p(1))\\
& = & Cn^{-c/6}(\log n)^2(1+o_p(1)) \to 0
\eeaa
 in probability as $n\to\infty$ by (\ref{donkey_elephant}) again. We then get the second conclusion of the lemma.
\end{proof}
Finally we are ready to prove Theorem \ref{vase_flower}.
\begin{proof}[Proof of Theorem \ref{vase_flower}]
Let $a$ be as in (\ref{learn}). Set
\beaa
& & U_n=\frac{\pi}{ \sqrt{6n} } m - \log \frac{ \sqrt{6n} }{\pi};\\
& & V_n=a-n^{-3/2}\sum_{j=1}^m k_j^2;\ \ \  W_n=\frac{a}{2}-n^{-3/2}\sum_{j=1}^m jk_j.
\eeaa
By (\ref{seabed}) and Lemma \ref{happy_uniform}, $U_n$ converges weakly to cdf $F_{\mu}(v) = e^{-e^{-v}}$ as $n\to \infty$, and both $V_n$ and $W_n$ converge to $0$ in probability. Solving $m$, $\sum_{j=1}^m k_j^2$  and  $\sum_{j=1}^m jk_j$ in terms of $U_n$, $V_n$ and  $W_n$, respectively, and substituting them for the corresponding terms of  $\lambda_{\kappa}$ in Lemma \ref{vase_flower}, we get
\beaa
\lambda_{\kappa}
& = & -\frac{\alpha}{2}n +nm+\sum_{j=1}^m (\frac{\alpha}{2}k_j-j)k_j\\
& = & -\frac{\alpha}{2}n +n\big(U_n+ \log \frac{ \sqrt{6n} }{\pi}\big)\cdot \frac{\sqrt{6n}}{\pi} + \frac{\alpha}{2}(a-V_n)n^{3/2}- (\frac{a}{2} - W_n)n^{3/2}.
\eeaa
Therefore,
\bea\lbl{fruit_sky}
c\frac{\lambda_{\kappa}}{n^{3/2}}-\log \frac{\sqrt{n}}{c}=U_n+ \big(\frac{\alpha-1}{2}\big)ac -\frac{c\alpha}{2}V_n + cW_n + o(1)
\eea
as $n\to\infty$. We finally evaluate $a$  in (\ref{learn}). Indeed, by (\ref{little_brother}), the Taylor expansion and integration by parts,
\beaa
(ac)\cdot c^2
&= & \int_0^1\frac{\log^2 (1-t)}{t}\,dt\\
& = & 2 \int_{0}^1 \frac{1}{t} \log t \log (1-t)\,dt\\
& = & -2\int_0^1 \frac{1}{t} \log t \sum_{n=1}^\infty \frac{t^n}{n}\,dt
 =  -2\sum_{n=1}^{\infty} \frac{1}{n} \int_0^1 t^{n-1} \log t\,dt\\
& = & 2\sum_{n=1}^\infty \frac{1}{n^3} = 2\zeta(3).
\eeaa
This and (\ref{fruit_sky}) prove the theorem  by the Slutsky lemma.
\end{proof}

\subsection{Proof of Theorems \ref{difficult_easy}}\label{Proof_difficult_easy}
\begin{proof}[Proof of Theorem \ref{difficult_easy}] For a partition $\kappa$ and its conjugate $\kappa'$, Frobenius (1900) shows that
$$\frac{a(\kappa') - a(\kappa)}{\binom{n}{2}} = \frac{\chi^{\kappa}_{(2,1^{n-2})}}{\dim(\kappa)},$$
where $\chi^{\kappa}_{(2,1^{n-2})}$ is the value of $\chi^{\kappa}$, the irreducible character of $\mathcal{S}_n$ associated to $\kappa$, on the conjugacy class indexed by $(2,1^{n-2}) \vdash n$.

By Theorem 6.1 from Ivanov and Olshanski (2001) for the special case
\beaa
p_2^{\#^{(n)}}(\kappa): = n(n-1) \frac{\chi^{\kappa}_{(2,1^{n-2})}}{\dim(\kappa)}
\eeaa
or Theorem 1.2 from Fulman (2004), we have
\beaa
\frac{a(\kappa')-a(\kappa)}{n} \to N\big(0, \frac{1}{2}\big)
\eeaa
weakly as $n\to\infty$. It is known from  Baik {\it et al}. (1999), Borodin {\it et al}. (2000),  Johannson (2001) and Okounkov (2000) that
\bea\lbl{green:white}
\frac{k_1-2\sqrt{n}}{n^{1/6}} \to F_2\ \ \mbox{and}\ \ \frac{m-2\sqrt{n}}{n^{1/6}} \to F_2
\eea
weakly as $n\to\infty$, where $F_2$ is as in (\ref{tai}). The $k_1$ and $m$ have the same limiting distribution in \eqref{green:white}, since $k_1$ and $m$ are duals under transposition, and the distribution stays the same under transposition.
Therefore, by using (\ref{Jingle}) for the case $\alpha =1$,
\begin{equation*}
\begin{split}
\frac{\lambda_{\kappa} - 2n^{3/2}}{n^{7/6}} &= \frac{n(m-1) + a(\kappa')-a(\kappa) - 2n^{3/2}}{n^{7/6}}\\
&=\frac{m-2\sqrt{n}}{n^{1/6}} - n^{-1/6} + \frac{a(\kappa')-a(\kappa)}{n^{7/6}}
\end{split}
\end{equation*}
converges weakly to $F_2$ as $n\to\infty$, where $F_2$ is as in (\ref{tai}).
\end{proof}

\subsection{Proof of Theorem \ref{thm:LLLplan}}\lbl{non_stop_thousand}
The proof of Theorem \ref{thm:LLLplan} is involved. The reason is that, when $\alpha=1$, the term  $a(\kappa')-a(\kappa)$ is negligible as shown in the proof of Theorem \ref{difficult_easy} . When $\alpha \ne 1$, reviewing  (\ref{Jingle}), it will be seen next that the term $a(\kappa')\alpha-a(\kappa)$, under the Plancherel measure,  is much larger and contributes to $\lambda_{\kappa}$ essentially.

\begin{figure}[!ht]
 \begin{center}
   \includegraphics[width=7cm]{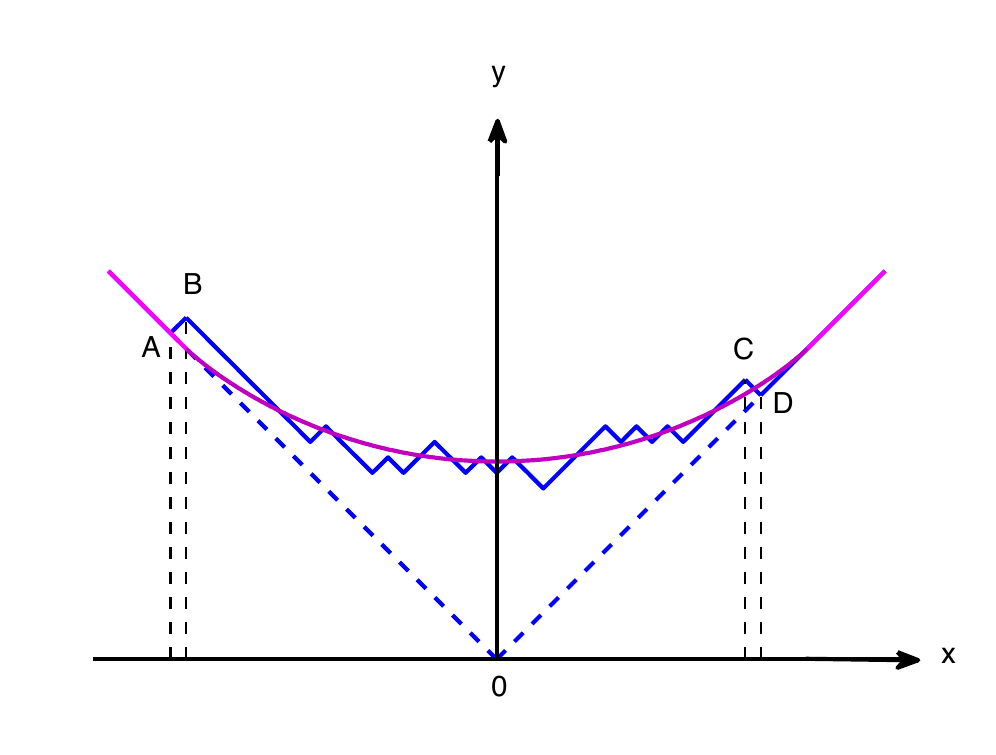}
   \caption{The ``zig-zag" curve is the graph of $y=g_{\kappa}(x)$ and the smooth one is $y=\Omega(x)$. Facts:  $A=(-\frac{m}{\sqrt{n}}, \frac{m}{\sqrt{n}})$,  $D=(\frac{k_1}{\sqrt{n}}, \frac{k_1}{\sqrt{n}})$, and $g_{\kappa}(x)=\Omega(x)$ if $x\geq \max\{\frac{k_1}{\sqrt{n}}, 2\}$ or $x\leq -\max\{\frac{m}{\sqrt{n}}, 2\}$.}
   \label{fig:kerov}
\end{center}
\end{figure}

 We first recall some notation. Let $\kappa=(k_1, k_2, \cdots, k_m)$ with $k_m\geq 1$ be a partition of $n$.
Set
coordinates $u$ and $v$ by
\bea\lbl{rotate_135}
u=\frac{j-i}{\sqrt{n}}\ \ \mbox{and}\ \ v=\frac{i+j}{\sqrt{n}}.
\eea
This is the same as flipping and then rotating the diagram of $\kappa$ counter clockwise $135^{\circ}$  and scaling  it by
a factor of $\sqrt{n/2}$ so that the area of the new diagram is equal to $2$. Denote by $g_{\kappa}(x)$ the boundary curve of the new Young diagram. See such a graph as in Figure \ref{fig:kerov}. It follows that  $g_{\kappa}(x)$ is a Lipschitz function for all $x\in \mathbb{R}.$

For a piecewise smooth and compactly supported function $h(x)$ defined on $\mathbb{R}$, its Sobolev norm is given by
\bea\lbl{Sobolev}
\|h\|_{\theta}^2=\iint_{\mathbb{R}^2}\Big(\frac{h(s)-h(t)}{s-t}\Big)^2\,dsdt.
\eea

Let $\kappa=(k_1, k_2, \cdots, k_m)$ with $k_m\geq 1$ be a partition of $n$. For $x\geq 0$, the notation $\lceil x\rceil$ stands for the least positive integer greater than or equal to $x$.
Define 
\bea\lbl{Logan}
f_{\kappa}(x)=\frac{1}{\sqrt{n}}k_{\lceil \sqrt{n}x \rceil}, \ \ \ \ \ x\geq 0.
\eea

Recall from \eqref{jilin} that $\Omega(x) = \frac{2}{\pi}(x\arcsin\frac{x}{2} + \sqrt{4-x^2})$ for $|x|\leq 2$ and $|x|$ otherwise. The following is a large deviation bound on a rare event under the Plancherel measure.
\begin{lemma}\lbl{season_mate}  Define $L_{\kappa}(x)=\frac{1}{2}g_{\kappa}(2x)$ and $\bar{\Omega}(x)=\frac{1}{2}\Omega(2x)$ for  $x\in\mathbb{R}$. Then for any $n\geq 2$ and any subset $\mathcal{F}$ of the partitions of $n,$ $$P(\mathcal{F}) \leq  \exp\big\{C\sqrt{n}-n  \inf_{\kappa\in \mathcal{F}}I(\kappa)\big\},$$
where $C>0$ is an absolute constant  and
\bea\lbl{rate_function_5}
I(\kappa)=\|L_{\kappa}-\bar{\Omega}\|_{\theta}^2\, -\, 4\int_{|s|>1}(L_{\kappa}(s)-\bar{\Omega}(s))\cosh^{-1} |s|\,ds.
\eea
\end{lemma}
\begin{proof}[Proof of Lemma \ref{season_mate}]
For any non-increasing function $F(x)$ defined on $(0, \infty)$ such that  $\int_{\mathbb{R}}F(x)\,dx=1$, define
\beaa
\theta_F=1+2\int_0^{\infty}\int_0^{F(x)}\log \big(F(x)+F^{-1}(y)-x-y\big)\,dy\,dx
\eeaa
where $F^{-1}(y)=\inf\{x\in \mathbb{R};\, F(x)\leq y\}$.
According to (1.8) from Logan and Shepp (1977), $P(\kappa)\leq C \sqrt{n}\cdot\exp\big\{-n \theta_{f_{\kappa}}\big\}$
for all $n\geq 2$, where $C$ is a numerical constant and $f_{\kappa}$ is defined as in (\ref{Logan}). By the Euler-Hardy-Ramanujan formula, $p(n)$, the total number of partitions of $n$, satisfies that
\bea\lbl{Raman}
p(n)\sim \frac{1}{4\sqrt{3}\,n}\cdot \exp\Big\{\frac{2\pi}{\sqrt{6}}\sqrt{n}\Big\}
\eea
as $n\to\infty$. Thus, for any  subset $\mathcal{F}$ of the partitions of $n,$ we have
\beaa
P(\mathcal{F}) &\leq & C p(n)\cdot\sqrt{n}\exp\Big\{-n  \inf_{\kappa\in \mathcal{F}}\theta_{f_{\kappa}}\Big\} \nonumber\\
& \leq & C' \exp\Big\{C'\sqrt{n}-n  \inf_{\kappa\in \mathcal{F}}\theta_{f_{\kappa}}\Big\}
\eeaa
where $C'$ is another numerical constant independent of $n$.  For the curve $y=f_\kappa(x)$ in \eqref{Logan}, consider the following transform
\beaa
X=\frac{x-y}{2}\ \ \mbox{and}\ \ Y=\frac{x+y}{2}.
\eeaa
We name the new curve by $y=L_{f_\kappa}(x)$. By (\ref{rotate_135})  and the definition  $L_{\kappa}(x)=\frac{1}{2}g_{\kappa}(2x)$, we have $L_{f_{\kappa}}(x)=L_{\kappa}(-x)$ for all $x\in \mathbb{R}.$ By Lemmas 2, 3 and 4  from Kerov (2003),
\beaa
\theta_{f_{\kappa}} &=& \|L_{f_{\kappa}}-\bar{\Omega}\|_{\theta}^2 + 4\int_{|s|>1}(L_{f_{\kappa}}(s)-\bar{\Omega}(s))\cosh^{-1} |s|\,ds\\
& = & \|L_{\kappa}-\bar{\Omega}\|_{\theta}^2 -4\int_{|s|>1}(L_{\kappa}(s)-\bar{\Omega}(s))\cosh^{-1} |s|\,ds
\eeaa
considering $\Omega(x)$ is an even function. We then get the desired result.
\end{proof}

The next lemma says that the second term on the right hand side of (\ref{rate_function_5}) is small for almost all partitions.
\begin{lemma}\lbl{rate_function} Let $L_{\kappa}(x)$ and $\bar{\Omega}(x)$ be as in Lemma \ref{season_mate}. Let $\{t_n>0;\, n\geq 1\}$ satisfy $t_n\to\infty$ and $t_n=o(n^{1/3})$ as $n\to\infty.$ Set $H_n=\{\kappa=(k_1, \cdots, k_m)\vdash n;\, k_m\geq 1,\, 2\sqrt{n}-t_nn^{1/6} \leq m,\, k_1 \leq 2\sqrt{n}+t_nn^{1/6}\}$. Then, as $n\to\infty$, $P(H_n)\to 1$ and
\bea\lbl{five_words}
\int_{|s|>1}(L_{\kappa}(s)-\bar{\Omega}(s))\cosh^{-1} |s|\,ds\cdot\, I_{H_n}=O(n^{-2/3}t_n^2).
\eea
\end{lemma}
\begin{proof}[Proof of Lemma \ref{rate_function}] Since $m$ and $k_1$ have the same probability distribution under the Plancherel measure, by (\ref{green:white}),
$\lim_{n\to\infty}P(H_n)=1.$ Review the definitions of $L_{\kappa}$ and $\bar{\Omega}$ in Lemma \ref{season_mate}. Trivially,
\beaa
\mbox{LHS of } (\ref{five_words})=\frac{1}{4}\int_{|x|>2}(g_{\kappa}(x)-\Omega(x))\cosh^{-1} \frac{|x|}{2}\,dx\cdot\, I_{H_n}.
\eeaa
By definition, $g_{\kappa}(x)=\Omega(x)$ if $x\geq \frac{k_1}{\sqrt{n}}\vee 2$ or $x\leq -\big(\frac{m}{\sqrt{n}}\vee 2\big).$ It follows that
\bea
& & \mbox{LHS of } (\ref{five_words}) \nonumber\\
& \leq & C_n\cdot \Big[\int_2^{2+n^{-1/3}t_n}\big|g_{\kappa}(x)-\Omega(x)\big|\,dx+ \int^{-2}_{-2-n^{-1/3}t_n}\big|g_{\kappa}(x)-\Omega(x)\big|\,dx\Big] \lbl{strive}
\eea
where
\beaa
C_n &= & \sup\Big\{\cosh^{-1}\frac{|x|}{2};\, -(\frac{m}{\sqrt{n}}\vee 2) \leq x \leq \frac{k_1}{\sqrt{n}} \vee 2\Big\}\cdot\, I_{H_n}\\
& \leq & \sup\Big\{\cosh^{-1}\frac{|x|}{2};\, -3\leq x \leq 3\Big\}<\infty
\eeaa
as $n$ is sufficiently large. Now
\bea
& & \int_2^{2+n^{-1/3}t_n}\big|g_{\kappa}(x)-\Omega(x)\big|\,dx\cdot\, I_{H_n} \nonumber\\
& \leq & n^{-1/3}t_n\cdot \max\Big\{\big|g_{\kappa}(x)-\Omega(x)\big|;\, 2\leq x \leq 2+n^{-1/3}t_n\Big\}\cdot\, I_{H_n}.\lbl{Lie}
\eea
By the triangle inequality, the Liptschitz property of $g_{\kappa}(x)$ and the fact $\Omega(x)=|x|$ for $|x|\geq 2$, we see
\beaa
 \big|g_{\kappa}(x)-\Omega(x)\big|
&\leq & \big|g_{\kappa}(x)-g_{\kappa}(2+2n^{-1/3}t_n)\big| + \big|g_{\kappa}(2+2n^{-1/3}t_n)-\Omega(x)\big|\\
& \leq & \big|x-(2+2n^{-1/3}t_n)\big|+ \big|2+2n^{-1/3}t_n-x\big|\\
& \leq & 2\big[(2+2n^{-1/3}t_n)-x\big] \leq 4n^{-1/3}t_n
\eeaa
for  $2\leq x \leq 2+n^{-1/3}t_n$ and $\kappa \in H_n$ whence $g_{\kappa}(2+2n^{-1/3}t_n) = 2+2n^{-1/3}t_n$. This and (\ref{Lie}) imply that the first integral in (\ref{strive}) is dominated by $O(n^{-2/3}t_n^2)$. By the same argument, the second integral in (\ref{strive}) has the same upper bound. Then the conclusion follows.
\end{proof}

To prove Lemma \ref{wisdom}, we need to  examine $g_{\kappa}(x)$ more closely.  For $(k_1, k_2, \ldots, k_m) \vdash n$, assume
\bea\lbl{give_word}
& & k_1=\cdots =k_{l_1}> k_{l_1+1}=\cdots=k_{l_2}>\cdots > k_{l_{p-1}+1}=\cdots=k_m \geq 1 \ \ \mbox{with}\nonumber\\
& & 0=l_0<l_1<\cdots < l_p=m
\eea
for some $p\geq 1.$  To ease notation, let $\bar{k}_i=k_{l_i}$ for $i=1,2,\cdots, p$ and $\bar{k}_{p+1}=0$. So the partition $\kappa$ is determined by $\{\bar{k}_i, l_i\}$'s.
It is easy to see that the corners (see, e.g., points $A, B, C,D$ in Figure \ref{fig:kerov}) sitting on the curve of $y=g_{\kappa}(x)$ listed  from the leftmost to the rightmost in order are
\beaa
\big(-\frac{l_p}{\sqrt{n}}, \frac{l_p}{\sqrt{n}}\big),  \cdots,\big(\frac{\bar{k}_i-l_i}{\sqrt{n}},\frac{\bar{k}_i+l_i}{\sqrt{n}}\big),  \big(\frac{\bar{k}_{i+1}-l_i}{\sqrt{n}}, \frac{\bar{k}_{i+1}+l_i}{\sqrt{n}}\,\big), \cdots, \big(\frac{\bar{k}_1}{\sqrt{n}}, \frac{\bar{k}_1}{\sqrt{n}}\,\big)
\eeaa
for $i=1, 2, \cdots, p$.
As a consequence,
\bea\lbl{Russian}
g_{\kappa}(x)=\begin{cases}
\frac{2\bar{k}_{i}}{\sqrt{n}} -x, \ \ \text{if $\frac{\bar{k}_{i}-l_i}{\sqrt{n}} \leq x \leq \frac{\bar{k}_{i}-l_{i-1}}{\sqrt{n}}$;}\\
\frac{2l_{i}}{\sqrt{n}} +x, \ \ \text{if $\frac{\bar{k}_{i+1}-l_{i}}{\sqrt{n}} \leq x \leq \frac{\bar{k}_{i}-l_{i}}{\sqrt{n}}$}
\\
\end{cases}
\eea
for all $1\leq i \leq p$, and $g_{\kappa}(x)=|x|$ for other $x\in \mathbb{R}$. In particular, taking $i=1$ and $p$, respectively, we get
\beaa
g_{\kappa}(x)=\begin{cases}
\frac{2k_1}{\sqrt{n}} - x, \ \ \text{if $\frac{k_1-l_1}{\sqrt{n}} \leq x \leq \frac{k_1}{\sqrt{n}}$};\\
\frac{2m}{\sqrt{n}} + x, \ \ \text{if $-\frac{m}{\sqrt{n}} \leq x \leq \frac{k_m-m}{\sqrt{n}}$}
\end{cases}
\eeaa
for $l_0=0$, $l_p=m$, $\bar{k}_1=k_1$,  and $\bar{k}_{p}=k_m$.


We need to estimate  $\sum_{i=1}^{m}ik_i$ in the proof of Theorem \ref{thm:LLLplan}. The following lemma links it  to  $g_{\kappa}(x)$.  We will  then be able to evaluate the sum through Kerov's central limit theorem (Ivanov and Olshanski, 2001).

\begin{lemma}\lbl{wisdom} Let $\kappa=(k_1, k_2, \cdots, k_m) \vdash n$ with $k_m\geq 1$ and $g_{\kappa}(x)$ be as in (\ref{Russian}). Then
\beaa
\sum_{i=1}^{m}ik_i
&= & \frac{1}{8}n^{3/2}\int_{-m/\sqrt{n}}^{k_1/\sqrt{n}}(g_{\kappa}(x)-x)^2\,dx -\frac{1}{6}m^3 + \frac{1}{2}n.
\eeaa
\end{lemma}
\begin{proof}[Proof of Lemma \ref{wisdom}]
Easily,
\bea
& & \int_{-m/\sqrt{n}}^{k_1/\sqrt{n}}(g_{\kappa}(x)-x)^2\,dx \nonumber\\
& = & \sum_{i=1}^p\int_{(\bar{k}_i-l_{i})/\sqrt{n}}^{(\bar{k}_i-l_{i-1})/\sqrt{n}}(g_{\kappa}(x)-x)^2\,dx  +
\sum_{i=1}^{p}\int_{\frac{\bar{k}_{i+1}-l_{i}}{\sqrt{n}}}^{\frac{\bar{k}_{i}-l_{i}}{\sqrt{n}}}(g_{\kappa}(x)-x)^2\,dx. \lbl{long_sum}
\eea
By (\ref{Russian}), the slopes of $g_{\kappa}(x)$ in the first sum of (\ref{long_sum}) are equal to $-1$. Hence, it   is equal to
\beaa
 4\sum_{i=1}^p\int_{(\bar{k}_i-l_{i})/\sqrt{n}}^{(\bar{k}_i-l_{i-1})/\sqrt{n}}\big(\frac{\bar{k}_i}{\sqrt{n}}-x\big)^2\,dx
& = & 4\sum_{i=1}^p\int_{l_{i-1}/\sqrt{n}}^{l_{i}/\sqrt{n}}t^2\,dt\\
& = & 4\int_{l_{0}/\sqrt{n}}^{l_{p}/\sqrt{n}}t^2\,dt=\frac{4m^3}{3n^{3/2}}
\eeaa
because $l_0=0$ and $l_p=m.$ In the second sum in (\ref{long_sum}), $g_{\kappa}(x)$ has slopes equal to $1$. As a consequence, it is identical to
\beaa
\sum_{i=1}^{p}\int_{\frac{\bar{k}_{i+1}-l_{i}}{\sqrt{n}}}^{\frac{\bar{k}_{i}-l_{i}}{\sqrt{n}}}\frac{4l_i^2}{n}\,dx
=\frac{4}{n^{3/2}}\sum_{i=1}^{p}(\bar{k}_i-\bar{k}_{i+1})l_i^2.
\eeaa
In summary,
\bea\lbl{father_home}
\int_{-m/\sqrt{n}}^{k_1/\sqrt{n}}(g_{\kappa}(x)-x)^2\,dx =\frac{4m^3}{3n^{3/2}} + \frac{4}{n^{3/2}}\sum_{i=1}^{p}(\bar{k}_i-\bar{k}_{i+1})l_i^2.
\eea
Now, let us evaluate the sum. Set $k_{j}=0$ for $j>m$ for convenience  and $\Delta_i=k_i-k_{i+1}$ for $i=1,2,\cdots.$ Then $\Delta_i=0$ unless $i=l_1, \cdots, l_p$.
Observe
\beaa
\sum_{i=1}^{\infty}ik_i=\sum_{i=1}^{\infty}i\sum_{j=i}^{\infty}\Delta_j&=&\sum_{j=1}^{\infty}\Delta_j\sum_{i=1}^{j}i\\
& = & \frac{1}{2}\sum_{j=1}^{\infty}j^2\Delta_j+ \frac{1}{2}\sum_{j=1}^{\infty}j\Delta_j.
\eeaa
Furthermore,
\beaa
\sum_{j=1}^{\infty}j\Delta_j= \sum_{j=1}^{\infty}\sum_{i=1}^j\Delta_j=\sum_{i=1}^{\infty}\sum_{j=i}^{\infty}\Delta_j=\sum_{i=1}^{\infty}k_i=n.
\eeaa
The above two assertions say that $\sum_{j=1}^{\infty}j^2\Delta_j=-n+2\sum_{i=1}^{\infty}ik_i$. Now,
\beaa
\sum_{j=1}^{\infty}j^2\Delta_j=\sum_{i=1}^{p}l_i^2(k_{l_i}-k_{l_i+1})=\sum_{i=1}^{p}l_i^2(\bar{k}_i-\bar{k}_{i+1})
\eeaa
by the fact $k_{l_i+1}=k_{l_{i+1}}=\bar{k}_{i+1}$ from (\ref{give_word}). This together with (\ref{father_home}) shows
\beaa
\int_{-m/\sqrt{n}}^{k_1/\sqrt{n}}(g_{\kappa}(x)-x)^2\,dx=\frac{4m^3}{3n^{3/2}}+ \frac{4}{n^{3/2}}\big(-n+2\sum_{i=1}^{\infty}ik_i \big).
\eeaa
Solve this equation to get
\beaa
\sum_{i=1}^{\infty}ik_i=\frac{1}{8}n^{3/2}\int_{-m/\sqrt{n}}^{k_1/\sqrt{n}}(g_{\kappa}(x)-x)^2\,dx -\frac{1}{6}m^3 + \frac{1}{2}n.
\eeaa
The proof is complete.
\end{proof}


Under the Plancherel measure, both $m/\sqrt{n}$ and $k_1/\sqrt{n}$ go to $2$ in probability. In lieu of  this fact, the next lemma writes the integral in Lemma \ref{wisdom} in a slightly cleaner form. The main tools of the proof are the Tracy-Widom law of the largest part of a random partition, the large deviations and Kerov's cental limit theorem.

\begin{lemma}\lbl{pen_name} Let $g_{\kappa}(x)$ be as in (\ref{Russian}) and set
\beaa
Z_n=\int_{-m/\sqrt{n}}^{k_1/\sqrt{n}}(g_{\kappa}(x)-x)^2\,dx-\int_{-2}^{2}(\Omega(x)-x)^2\,dx
\eeaa
where $\Omega(x)$ is as in (\ref{jilin}). Then, for any $\{a_n>0;\, n\geq 1\}$ with $\lim_{n\to\infty}a_n=\infty$, we have  $$\frac{n^{1/4}}{a_n} Z_n \to 0$$ in probability as $n\to\infty$.
\end{lemma}
\begin{proof}[Proof of Lemma \ref{pen_name}] Without loss of generality, we assume
\bea\lbl{song_spring}
a_n=o(n^{1/4})
\eea
as $n\to\infty$. Set
\beaa
Z_n'=\int_{-2}^{2}(g_{\kappa}(x)-x)^2\,dx-\int_{-2}^{2}(\Omega(x)-x)^2\,dx.
\eeaa
Write
\bea\lbl{news_paper}
\frac{n^{1/4}}{a_n}Z_n= \frac{n^{1/4}}{a_n}Z_n'+ \frac{1}{n^{1/12} a_n}R_{n,1}+\frac{1}{n^{1/12} a_n}R_{n,2},
\eea
where
\beaa
& & R_{n,1}:=n^{1/3}\int_{-m/\sqrt{n}}^{-2}(g_{\kappa}(x)-x)^2 \,dx;\\
& & R_{n,2}:=n^{1/3}\int_{2}^{k_1/\sqrt{n}}(g_{\kappa}(x)-x)^2 \,dx.
\eeaa
We will show the  three terms on the right hand side of (\ref{news_paper}) go to zero in probability.

{\it Step 1}. We will prove a stronger result that both $R_{n,1}$ and $R_{n,2}$ are of order of $O_p(1)$ as $n\to\infty.$ We start with $R_{n,1}$. The proof essentially bounds the integrand of $R_{n,1}$ for $-m/\sqrt{n}\le x \le -2$, which can be achieved via \eqref{green:white} and the following result. By Theorem 5.5 from Ivanov  and Olshanski (2001),
\bea\lbl{Minnesota_lib_tc}
\delta_n:= \sup_{x\in \mathbb{R}}|g_{\kappa}(x)-\Omega(x)|\to 0
\eea
in probability as $n\to\infty$, where $\Omega(x)$ is defined in (\ref{jilin}). Observe that
\beaa
\frac{1}{2}|g_{\kappa}(x)-x|^2 \leq \delta_n^2+(\Omega(x)-x)^2
\eeaa
for each $x\in \mathbb{R}.$  Denote $C=\sup_{-3\leq x\leq 0}(\Omega(x)-x)^2$ and $$C_n=\sup_{-m/\sqrt{n}\le x \le -2}(\Omega(x)-x)^2.$$ Then $P(C_n>2C)\leq P(\frac{m}{\sqrt{n}}>3)\to 0$ by (\ref{green:white}). Therefore, $C_n=O_p(1)$. It follows that
\bea
|R_{n,1}|  \leq  2n^{1/3} \big|\frac{m}{\sqrt{n}}-2\big|\cdot(\delta_n^2 +C_n)=O_p(1)\lbl{catalog}
\eea
 by (\ref{green:white}) again.
Similarly, $R_{n,2}=O_p(1)$ as $n\to\infty$.

In the rest of the proof, we only need to show $\frac{n^{1/4}}{a_n} Z_n'$ goes to zero in probability. This again takes several steps.

{\it Step 2}. In this step we will reduce $Z_n'$ to a workable form.  By the same argument as the one used in proving (\ref{catalog}),
we have
\beaa
Z_n' &= & \int_{-2}^2(g_{\kappa}(x)-\Omega(x)) (g_{\kappa}(x)-\Omega(x)+2(\Omega(x)-x))\,dx\\
& = & \int_{-2}^2|g_{\kappa}(x)-\Omega(x)|^2\,dx + \int_{-2}^2f_1(x)(g_{\kappa}(x)-\Omega(x)) \,dx\\
& \le & \int_{-2}^2|g_{\kappa}(x)-\Omega(x)|^2\,dx + \sqrt{\int_{-2}^2 f_1(x) \,dx} \cdot \sqrt{ \int_{-2}^2 |g_{\kappa}(x)-\Omega(x)|^2  \,dx}
\eeaa
where $f_1(x):=2(\Omega(x)-x)$ for all $x \in \mathbb{R}$, and the last inequality above follows from the Cauchy-Schwartz inequality. To show $\frac{n^{1/4}}{a_n} Z_n'$ goes to zero in probability, since $f_1(x)$ is a bounded function on $\mathbb{R}$, it suffices to prove
\bea\lbl{mellon}
Z_n'':&=& \frac{n^{1/2}}{a_n^2} \int_{-2}^2|g_{\kappa}(x)-\Omega(x)|^2\,dx \to 0
\eea
in probability by (\ref{song_spring}). Set
\bea\lbl{hapiness}
H_n &= & \Big\{\kappa=(k_1, \cdots, k_m)\vdash n;\, 2\sqrt{n}-n^{1/6}\log n \leq m,\, k_1 \leq 2\sqrt{n}+n^{1/6}\log n\ \mbox{and}\ \nonumber\\
& &   \ \ \ \ \ \ \ \ \ \ \ \ \ \ \ \ \ \ \ \ \ \ \ \ \ \ \ \ \ \ \ \ \ \ \ \ \ \ \ \ \ \ \ \big|n^{1/3}\int_{-2}^2(g_{\kappa}(x)-\Omega(x))\,ds\big|\leq 1 \Big\}.
\eea

{\it Step 3 }. We prove in this step that
\bea\lbl{garlic_sweet}
\lim_{n\to\infty}P(H_n^c)=0.
\eea
Note that $g_{\kappa}(s)=\Omega(s)=|s|$ if $s\geq \max\{\frac{k_1}{\sqrt{n}}, 2\}$ or $s\leq -\max\{\frac{m}{\sqrt{n}}, 2\}$. Also, the areas encircled by $t=|s|$ and $t=g_{\kappa}(s)$ and that by $t=|s|$ and $t=\Omega(s)$ are both equal to $2$; see Figure \ref{fig:kerov}. It is trivial to see that $\int_{a}^b(g_{\kappa}(s)-\Omega(s))\,du=\int_{\mathbb{R}}(g_{\kappa}(s)-\Omega(s))\,du=0$ for $a:= -\max\{\frac{m}{\sqrt{n}}, 2\}$ and $b:= \max\{\frac{k_1}{\sqrt{n}}, 2\}$. Define  $$h_{\kappa}(s)=g_{\kappa}(s)-\Omega(s).$$
We see
\beaa
-\int_{-2}^2h_{\kappa}(s)\,ds=\int_a^{-2}h_{\kappa}(s)\,ds + \int_{2}^{b}h_{\kappa}(s)\,ds.
\eeaa
Thus,
\beaa
 \big|n^{1/3}\int_{-2}^2h_{\kappa}(s)\,ds\big|
 &\leq & \big|n^{1/3}\int_a^{-2}h_{\kappa}(s)\,ds\big| + \big|n^{1/3}\int_{2}^bh_{\kappa}(s)\,ds\big| \nonumber\\
& \leq & 2n^{1/3}\max_{s\in \mathbb{R}}|h_{\kappa}(s)| \cdot (|a+2|+|b-2|).
\eeaa
From (\ref{Minnesota_lib_tc}), $\max_{s\in \mathbb{R}}|h_{\kappa}(s)|\to 0$ in probability. Further  $|a+2| \leq |\frac{m}{\sqrt{n}}-2|$ and $|b-2| \leq |\frac{k_1}{\sqrt{n}}-2|$.  By (\ref{green:white}) again, we obtain $n^{1/3}\int_{-2}^2h_{\kappa}(s)\,du \to 0$ in probability. This and the first conclusion of Lemma  \ref{rate_function} imply that $\lim_{n\to\infty}P(H_n^c)=0$.

{\it Step 4}. Review $H_n$ in (\ref{hapiness}) and the limit in (\ref{garlic_sweet}).
From the bound $P(Z_n''>\epsilon) \le P(H_n \cap \{Z_n''>\epsilon\}) + P(H_n^c)$,  we apply Lemma \ref{season_mate} for the set $\mathcal F = H_n \cap \{Z_n''> \epsilon\}$ for the first term on the RHS of the bound. 
It is seen from  Lemma \ref{season_mate} that there exists an absolute constant $C>0$ such that
\beaa
P(Z_n''>\epsilon) & \leq & e^{C\sqrt{n}-n \cdot \inf I(\kappa)} + P(H_n^c) \\
& = & e^{C\sqrt{n}-n \cdot \inf I(\kappa)} + o(1).
\eeaa
where $I(\kappa)$ is as in Lemma \ref{season_mate} and the infimum is taken over all $\kappa\in H_n \cap \{Z_n''> \epsilon\}$. We claim
\bea\lbl{tea_jilin_s}
n^{1/2}\cdot \inf_{\kappa\in H_n;\, Z_n''\geq \epsilon} I(\kappa)\to\infty
\eea
as $n\to\infty$. If this is true, we then obtain (\ref{mellon}), and the proof is completed. Review
\beaa
I(\kappa)=\|L_{\kappa}-\bar{\Omega}\|_{\theta}^2 - 4\int_{|s|>1}(L_{\kappa}(s)-\bar{\Omega}(s))\cosh^{-1} |s|\,ds.
\eeaa
Lemma \ref{rate_function} says that the last term above is of order $O(n^{-2/3}(\log n)^2)$ as $\kappa\in H_n$ by taking $t_n=\log n$.
To get (\ref{tea_jilin_s}), it suffices to show
\bea\lbl{waterfall}
n^{1/2}\cdot \inf_{\kappa\in H_n;\, Z_n''\geq \epsilon}\|L_{\kappa}-\bar{\Omega}\|_{\theta}^2 \to\infty
\eea
as $n\to\infty$.  By the definitions of $L_{\kappa}$ and $\bar{\Omega}$, we see from (\ref{Sobolev}) that
\beaa
\|L_{\kappa}-\bar{\Omega}\|_{\theta}^2
& \geq & \frac{1}{4}\int_{-2}^2\int_{-2}^2\Big(\frac{h_{\kappa}(s)-h_{\kappa}(t)}{s-t}\Big)^2\,dsdt\\
& \geq & \frac{1}{4^3}\int_{-2}^2\int_{-2}^2(h_{\kappa}(s)-h_{\kappa}(t))^2\,dsdt\\
& = &  \frac{1}{4} E(h_{\kappa}(U)-h_{\kappa}(V))^2
\eeaa
where  $U$ and $V$ are independent random variables with the uniform distribution on $[-2, 2].$ By the Jensen inequality, the last integral is bounded below by $E(h_{\kappa}(U)-Eh_{\kappa}(V))^2=E[h_{\kappa}(U)^2]-[Eh_{\kappa}(V)]^2$. Consequently,
\beaa
\|L_{\kappa}-\bar{\Omega}\|_{\theta}^2
& \geq & \frac{1}{16}\int_{-2}^2h_{\kappa}(u)^2\,du - \frac{1}{64}\Big(\int_{-2}^2h_{\kappa}(u)\,du\Big)^2 \\ 
& \geq & \frac{\epsilon}{16}n^{-1/2}\cdot a_n^2 -\frac{1}{64}n^{-2/3}
\eeaa
for $\kappa \in H_n \cap \{Z_n''\geq \epsilon\}$. This implies (\ref{waterfall}).
\end{proof}

With the above preparation we proceed to prove Theorem \ref{thm:LLLplan}.

\begin{proof}[Proof of Theorem \ref{thm:LLLplan}] By Lemma \ref{theater},
\beaa
\lambda_{\kappa}=\big(m-\frac{\alpha}{2}\big)n +\sum_{i=1}^m (\frac{\alpha}{2}k_i-i)k_i.
\eeaa
Thus
\beaa
&&\frac{\lambda_{\kappa} -2n^{3/2}- (\alpha-1) (\frac{128}{27} \pi^{-2} )n^{3/2} }{ n^{5/4} \cdot a_n }\\
&=&\frac{m-2\sqrt{n}}{n^{1/4} \cdot a_n} - \frac{\alpha}{2n^{1/4} \cdot a_n} +\frac{\sum_{i=1}^m(\frac{\alpha}{2}k_i-i)k_i-(\alpha-1) ( \frac{128}{27} \pi^{-2} )n^{3/2} }{n^{5/4} \cdot a_n}.
\eeaa
We claim
\bea\lbl{cold_leave}
\frac{\sum_{i=1}^m(\frac{\alpha}{2}k_i-i)k_i-(\alpha-1) (\frac{128}{27} \pi^{-2} )n^{3/2} }{n^{5/4} \cdot a_n} \to 0
\eea
in probability as $n\to\infty.$ If this is true, by (\ref{green:white}), we  finish the proof. Now let us show (\ref{cold_leave}).

We first claim
\bea\lbl{we}
\frac{1}{n}\sum_{i=1}^m\big(\frac{1}{2}k_i-i\big)k_i \to  N\big(-\frac{1}{2}, \sigma^2\big)
\eea
for some $\sigma^2\in (0, \infty)$. To see why this is true, we get from (\ref{kernel_sea}) and Lemma \ref{theater} that \beaa
a(\kappa')-a(\kappa)=\frac{1}{2}n+\sum_{i=1}^m\big(\frac{1}{2}k_i-i\big)k_i.
\eeaa
By Theorem 1.2 from Fulman (2004), there is $\sigma^2\in (0, \infty)$ such that
\beaa
\frac{a(\kappa')-a(\kappa)}{n} \to N(0, \sigma^2)
\eeaa
weakly as $n\to\infty$. Then (\ref{we}) follows.

Second, from (\ref{green:white}), we know $\xi_n:=(m-2\sqrt{n})n^{-1/6}$ converges weakly to  $F_2$ as $n\to\infty.$ Write
\beaa
m^3 = (2\sqrt{n} + n^{1/6}\xi_n)^3=n^{1/2}\xi_n^3 + 6n^{5/6}\xi_n^2 + 12 n^{7/6}\xi_n + 8n^{3/2}.
\eeaa
This implies that $$\frac{m^3-8n^{3/2}}{n^{5/4}} \to 0$$ in probability as $n \to \infty$. Let $Z_n$ be as in Lemma \ref{pen_name} and $\Omega(x)$ as in (\ref{jilin}). It is seen from Lemmas \ref{wisdom} and \ref{pen_name} that
\beaa
\sum_{i=1}^m i k_i &=& \frac{1}{8}n^{3/2}\int_{-m/\sqrt{n}}^{k_1/\sqrt{n}}(g_{\kappa}(x)-x)^2\,dx -\frac{1}{6}m^3 + \frac{1}{2}n\\
&=&  \frac{1}{8}n^{3/2} \left( Z_n + \int_{-2}^2 (\Omega(x)-x)^2 \,dx \right)-\frac{1}{6}m^3 + \frac{1}{2}n
\eeaa
with $\frac{n^{1/4}}{8a_n} Z_n \to 0$ in probability as $n\to\infty.$ The last two assertions imply
\bea
& & \frac{1}{{n^{5/4} \cdot a_n}}\Big[\sum_{i=1}^mik_i - \frac{1}{8} n^{\frac{3}{2}}  \int_{-2}^2 (\Omega(x)-x)^2 \,dx + \frac{4}{3} n^{\frac{3}{2}}\Big]  \lbl{field_sleep}\\
& = & \frac{n^{1/4}}{8a_n} Z_n  - \frac{1}{6a_n}\cdot \frac{m^3-8n^{3/2}}{n^{5/4}} + \frac{1}{2a_n n^{1/4}}
\to 0\nonumber
\eea
in probability as $n\to\infty$. It is trivial and yet a bit tedious to verify
\bea\lbl{integral}
\int_{-2}^2 (\Omega(x)-x)^2 \,dx = \frac{32}{3} + \frac{1024}{27\pi^2}.
\eea
The calculation of \eqref{integral} is included in Appendix \ref{appendix:integral}. Plug this into (\ref{field_sleep}) to see
\bea\lbl{you}
\frac{\sum_{i=1}^mik_i - \frac{128}{27\pi^2} n^{3/2}}{n^{5/4} \cdot a_n}
\to 0
\eea
in probability as $n\to\infty$.

Third, observe
\beaa
\sum_{i=1}^m(\frac{\alpha}{2}k_i-i)k_i
= \alpha\sum_{i=1}^m\big(\frac{1}{2}k_i-i\big)k_i+(\alpha-1)\sum_{i=1}^mik_i.
\eeaa
Therefore
\beaa
& &\frac{\sum_{i=1}^m(\frac{\alpha}{2}k_i-i)k_i-(\alpha-1) (\frac{128}{27} \pi^{-2} )n^{3/2} }{n^{5/4} \cdot a_n}\\
&=& \alpha \frac{\sum_{i=1}^m\big(\frac{1}{2}k_i-i\big)k_i}{n^{5/4} \cdot a_n} + (\alpha-1) \frac{\sum_{i=1}^mik_i - ( \frac{128}{27} \pi^{-2} )n^{3/2}}{n^{5/4} \cdot a_n} \to 0
\eeaa
in probability by \eqref{we} and \eqref{you}. We finally arrive at (\ref{cold_leave}).
\end{proof}

\section{Appendix}\lbl{appendix:last}
In this section we will prove \eqref{mean_variance}, verify \eqref{integral} and derive the density functions of the random variable appearing in Theorem \ref{cancel_temple} for two cases. They are placed in three subsections.

\subsection{Proof of \eqref{mean_variance}}\label{appendix:mean_variance}
Recall $(2s-1)!!=1\cdot 3\cdots (2s-1)$ for integer $s\geq 1$. Set $(-1)!!=1$. The following is Lemma 2.4 from Jiang (2009).
\begin{lemma}\lbl{Jiang2009} Suppose $p\geq 2$ and $Z_1, \cdots, Z_p$ are i.i.d. random variables with $Z_1 \sim N(0,1).$ Define $U_i=\frac{Z_i^2}{Z_1^2 + \cdots + Z_p^2}$ for $1\leq i \leq p$. Let $a_1, \cdots, a_p$ be non-negative integers and $a=\sum_{i=1}^pa_i$. Then
\beaa
E\big(U_1^{a_1}\cdots U_p^{a_p}\big) = \frac{\prod_{i=1}^p(2a_i-1)!!}{\prod_{i=1}^a(p+2i-2)}.
\eeaa
\end{lemma}
\noindent\textbf{Proof of (\ref{mean_variance})}. Recall (\ref{pro_land}). Write $(r-1)s^2=\sum_{i=1}^r x_i^2-r\bar{x}^2$. In our case,
\beaa
& &  \bar{x} = \frac{1}{|\mathcal{P}_n(m)|}\sum_{\kappa\in \mathcal{P}_n(m)}\lambda_{\kappa}=E\lambda_{\kappa};\\
& & s^2=\frac{1}{|\mathcal{P}_n(m)|-1}\sum_{\kappa\in \mathcal{P}_n(m)}(\lambda_{\kappa}-\bar{x})^2 \sim E(\lambda_{\kappa}^2)- (E\lambda_{\kappa})^2
\eeaa
as $n\to\infty$, where $E$ is the expectation about the uniform measure on $\ml{P}_n(m)'$. Therefore,
\bea\lbl{interview_Jin}
\frac{\bar{x}}{n^2}=\frac{E\lambda_{\kappa}}{n^2}\ \ \mbox{and}\ \ \frac{s^2}{n^4}\sim E\Big(\frac{\lambda_{\kappa}}{n^2}\Big)^2 - \Big(\frac{E\lambda_{\kappa}}{n^2}\Big)^2.
\eea
From Lemma \ref{theater}, we see a trivial bound that $0\leq \lambda_{\kappa}/n^2\leq 1+\frac{\alpha}{2}m$ for each partition $\kappa =(k_1, \cdots, k_m)\vdash n$ with $k_m\geq 1.$ By Theorem \ref{cancel_temple}, under $\ml{P}_n'(m)$,
\beaa
\frac{\lambda_{\kappa}}{n^2}\to \frac{\alpha}{2}\cdot Y\ \ \mbox{and}\ \ Y:=\frac{\xi_1^2+\cdots + \xi_m^2}{(\xi_1+\cdots + \xi_m)^2}
\eeaa
as $n\to\infty$, where  $\{\xi_i;\, 1\leq i \leq m\}$ are i.i.d. random variables with density  $e^{-x}I(x\geq 0)$. By bounded convergence theorem and (\ref{interview_Jin}),
\bea\lbl{do_tired}
\frac{\bar{x}}{n^2}\to \frac{\alpha}{2} EY\ \ \mbox{and}\ \ \frac{s^2}{n^4} \to \frac{\alpha^2}{4} [E(Y^2)-(EY)^2]
\eea
as $n\to\infty$. Now we evaluate $EY$ and $E(Y^2)$. Easily,
\bea
& & EY=m\cdot E\frac{\xi_1^2}{(\xi_1+\cdots + \xi_m)^2};\nonumber\\
& &  E(Y^2)=m\cdot E\frac{\xi_1^4}{(\xi_1+\cdots + \xi_m)^4} +m(m-1)\cdot E\frac{\xi_1^2\xi_2^2}{(\xi_1+\cdots + \xi_m)^4}.\lbl{moon_wind}
\eea
Let $Z_1, \cdots, Z_{2m}$ be i.i.d. random variables with $N(0, 1)$ and $U_i=\frac{Z_i^2}{Z_1^2 + \cdots + Z_{2m}^2}$ for $1\leq i \leq 2m$. Evidently, $(Z_1^2+Z_2^2)/2$ has density function $e^{-x}I(x\geq 0)$. Then,
\beaa
\Big(\frac{\xi_i}{\xi_1+\cdots + \xi_m}\Big)_{1\leq i \leq m}\ \  \mbox{and}\ \ \ (U_{2i-1}+U_{2i})_{1\leq i \leq m}
\eeaa
have the same distribution. Consequently, by taking $p=2m$ in Lemma \ref{Jiang2009},
\bea
EY &=& m\cdot E(U_1+U_2)^2 \nonumber\\
&= & 2m[E(U_1^2) +E(U_1U_2)] \nonumber\\
& = & 2m\big[\frac{3}{4m(m+1)} + \frac{1}{4m(m+1)}\big]=\frac{2}{m+1}.\lbl{easy_absolute}
\eea
Similarly,
\beaa
E\frac{\xi_1^4}{(\xi_1+\cdots + \xi_m)^4}
&=& E[(U_1+U_2)^4]\\
& = & 2E(U_1^4) +8 E(U_1^3U_2) + 6 E(U_1^2U_2^2)\\
& = & \frac{105}{8}\frac{1}{m(m+1)(m+2)(m+3)} + \frac{15}{2}\frac{1}{m(m+1)(m+2)(m+3)}\\
& & ~~~~~~~~~~~~~~~~~~~~~~~~~~~~~~~~~~~~~~~ + \frac{27}{8}\frac{1}{m(m+1)(m+2)(m+3)}\\
&= & \frac{24}{m(m+1)(m+2)(m+3)}
\eeaa
and
\beaa
 E\frac{\xi_1^2\xi_2^2}{(\xi_1+\cdots + \xi_m)^4} &= & E[(U_1+U_2)^2(U_3+U_4)^2]\\
 & = & 4E(U_1^2U_2^2) + 8E(U_1^2U_2U_3)+ 4E(U_1U_2U_3U_4)\\
 & = & \frac{9}{4}\frac{1}{m(m+1)(m+2)(m+3)} + \frac{3}{2}\frac{1}{m(m+1)(m+2)(m+3)}\\
 & &  ~~~~~~~~~~~~~~~~~~~~~~~~~~~~~~~~~~~~ + \frac{1}{4}\frac{1}{m(m+1)(m+2)(m+3)}\\
 &= & \frac{4}{m(m+1)(m+2)(m+3)}.
\eeaa
It follows from (\ref{moon_wind}) and (\ref{easy_absolute}) that
\beaa
& & E(Y^2)=\frac{4m+20}{(m+1)(m+2)(m+3)};\\
& & E(Y^2)-(EY)^2 = \frac{4m+20}{(m+1)(m+2)(m+3)}-\Big(\frac{2}{m+1}\Big)^2\\
& & ~~~~~~~~~~~~~~~~~~~~~= \frac{4m-4}{(m+1)^2(m+2)(m+3)}.
\eeaa
This and (\ref{do_tired}) say  that
\beaa
\frac{\bar{x}}{n^2}\to\frac{\alpha}{m+1}\ \ \mbox{and} \ \ \frac{s^2}{n^4}\to \frac{(m-1)\alpha^2}{(m+1)^2(m+2)(m+3)}.
\eeaa

\subsection{Verification of \eqref{integral}}\label{appendix:integral}
\begin{proof}[Verification of \eqref{integral}] Trivially,  $\Omega(x)$ in (\ref{jilin}) is an even function and $\Omega(x)' = \frac{2}{\pi} \arcsin\frac{x}{2}$ for $|x| < 2$. Then
\begin{equation*}
\begin{split}
 & \int_{-2}^2 (\Omega(x) - x)^2 \,dx  = \int_{-2}^2 \Omega(x)^2 \,dx + \int_{-2}^2 x^2 \,dx\\
& = x \cdot \Omega(x)^2 \Bigr|_{-2}^2 - \int_{-2}^2 x \cdot 2\Omega(x) \cdot \Omega(x)'  \,dx + \frac{x^3}{3} \Bigr|_{-2}^2\\
& = \frac{64}{3} - \frac{16}{\pi^2} \int_{0}^2 x \arcsin {\frac{x}{2}}  \cdot  ( x \arcsin\frac{x}{2} +\sqrt{4-x^2} )\,dx.
\end{split}
\end{equation*}
Now, set $x=2\sin \theta$ for $0\leq \theta \leq \frac
{\pi}{2}$, the above integral becomes
\bea\lbl{sine_cosine}
&&\int_{0}^{\frac{\pi}{2}} 2\theta \sin \theta (2\theta \sin \theta + 2\cos\theta) 2\cos \theta \, d\theta \nonumber\\
& = & 2\int_{0}^{\frac{\pi}{2}}(\theta \sin \theta+\theta \sin (3\theta)+\theta^2\cos\theta-\theta^2\cos (3\theta))\,d\theta
\eea
by trigonometric identities. It is easy to verify that
\beaa
& & \theta\sin \theta=(\sin\theta-\theta\cos\theta)';\ \  \ \  \theta\sin (3\theta)=\frac{1}{9}(\sin(3\theta)-3\theta\cos(3\theta))';\\
& &  \theta^2\cos\theta =(\theta^2\sin \theta +2\theta\cos \theta-2\sin \theta)';\\
  & & \theta^2\cos (3\theta)=\frac{1}{27}(9\theta^2\sin (3\theta) +6\theta\cos (3\theta)-2\sin (3\theta))'.
\eeaa
Thus, the term in (\ref{sine_cosine}) is equal to
\beaa
2\Big(1+(-\frac{1}{9})+ (\frac{\pi^2}{4}-2)-\frac{1}{27}(-\frac{9\pi^2}{4}+2)\Big)=\frac{2}{3}\pi^2-\frac{64}{27}.
\eeaa
It follows that
\beaa
\int_{-2}^2 (\Omega(x) - x)^2 \,dx  =\frac{64}{3} - \frac{16}{\pi^2}\Big(\frac{2}{3}\pi^2-\frac{64}{27}\Big)=\frac{32}{3}+\frac{1024}{27\pi^2}.
\eeaa
This completes the verification.
\end{proof}

\subsection{Derivation of density functions in Theorem \ref{cancel_temple}}\label{appendix:integral2}
In this section, we will derive explicit formulas for the limiting distribution in Theorem \ref{cancel_temple}. For convenience, we rewrite the conclusion as  $$\frac{2}{\alpha}\cdot  \frac{\lambda_\kappa}{n^2}\to \nu, $$ where  $\nu$ is different from $\mu$ in Theorem \ref{cancel_temple} by a factor of $\frac{2}{\alpha}$. We will only evaluate the cases $m=2, 3$.
We first state the conclusions and prove them afterwards.

\noindent{\it Case 1.} For $m=2$, the support of $\nu$ is $[\frac{1}{2}, 1]$ and the cdf of  $\nu$ is
\begin{equation}\label{eq:cdf-ru}
 F(t) =\sqrt{2t-1}
\end{equation}
for $t\in [\frac{1}{2}, 1]$. Hence the density function is given by
\beaa
 f(t)= \frac{1}{\sqrt{2t-1}},\ \ t\in [\frac{1}{2}, 1].
\eeaa

\noindent{\it Case 2.} For $m=3$, the support of $\nu$ is $[\frac{1}{3}, 1]$, and the cdf of  $\nu$ is
\bea\label{eq:cdf-haha}
 F(t) =
  \begin{cases}
   \frac{2}{\sqrt{3}} \pi (t-\frac{1}{3}), & \text{if } \frac{1}{3} \le t < \frac{1}{2}; \\
   \frac{2}{\sqrt{3}} \left( (t-\frac{1}{3}) (\pi - 3\arccos\frac{1}{\sqrt{6t-2}}) + \frac{\sqrt{6}}{2} \sqrt{t-\frac{1}{2}}\, \right),      & \text{if } \frac{1}{2} \le t < 1.
  \end{cases}
\eea
By differentiation, we get the density function
\beaa
 f(t)=
 \begin{cases}
   \frac{2}{\sqrt{3}} \pi, & \text{if } \frac{1}{3} \le t < \frac{1}{2}; \\
   \frac{2}{\sqrt{3}} \big( \pi - 3\arccos\frac{1}{\sqrt{6t-2}} \big),      & \text{if } \frac{1}{2} \le t \le 1.\\
  \end{cases}
\eeaa
The above are the two density functions claimed below the statement of Theorem \ref{cancel_temple}. Now we  prove them.

From a comment below Theorem \ref{cancel_temple}, the limiting law of $\frac{2}{\alpha}\cdot \frac{\lambda_\kappa}{n^2}$  is  the same as the distribution of $\sum_{i=1}^m Y_i^2$, where $(Y_1,\ldots,Y_m)$ has  uniform distribution over the set $$\mathcal{H}:=\Big\{(y_1,\ldots,y_{m}) \in [0,1]^{m};\, \sum_{i=1}^{m}y_i = 1 \Big\}.$$ By (\ref{silk_peer}) the volume of $\mathcal{H}$ is $\frac{\sqrt{m}}{(m-1)!}$. Therefore, the cdf of $\sum_{i=1}^m Y_i^2$ is
\begin{equation}\label{eq:dist}
F(t)=P\Big(\sum_{i=1}^m Y_i^2 \le t\Big) = \frac{(m-1)!}{\sqrt{m}} \cdot \text{volume  of }  \Big\{ \sum_{i=1}^m y_i^2 \le t\Big\}\cap \mathcal{H},\ \ t\geq 0.
\end{equation}

Denote $B_m(t):=\{ \sum_{i=1}^m y_i^2 \le t\} \subset \mathbb{R}^m$. Let $V(t)$ be the volume of $B_m(t) \cap \mathcal{H} $. We start with some facts for any  $m\geq 2$.

First, $V(t) = 0$ for $t < \frac{1}{m}$. In fact, if $(y_1, \cdots, y_m)\in B_m(t)\cap \mathcal{H}$, then
$$\frac{1}{m}=\frac{(\sum_{i=1}^m y_i)^2}{m} \le \sum_{i=1}^m y_i^2\le t.$$

Further, for $t>1$, $\mathcal{H}$ is inscribed in $B_m(t)$ and thus $V(t)  = \frac{\sqrt{m}}{(m-1)!}$. Now assume $1/m \le t \le 1$.

\medskip

\noindent{\it The proof of (\ref{eq:cdf-ru})}. Assume $m=2$. If $1/2 \le t \le 1$, then $\{(y_1,y_2) \in [0,1]^2: y_1 + y_2 =1 \}\cap\{ y_1^2 + y_2^2 \le t \}$ is a line segment. Easily,  the  endpoints of the line segment are $$\Big(\frac{1+\sqrt{2t-1}}{2}, \frac{1-\sqrt{2t-1}}{2}\Big) \quad \text{and} \quad \Big(\frac{1-\sqrt{2t-1}}{2}, \frac{1+\sqrt{2t-1}}{2}\Big),$$ respectively.
Thus $V(t) = \sqrt{2(2t-1)}.$
Therefore the conclusion follows directly from \eqref{eq:dist}.

\medskip

\noindent{\it The proof of (\ref{eq:cdf-haha})}. We first observe that as $t$ increases from $\frac{1}{3}$ to 1, the intersection $B_{3}(t)\cap \mathcal{H}$ expands and passes through $\mathcal{H}$ as $t$ exceeds some critical value $t_0$; see Figure \ref{fig:region}.

We claim that $t_0 = \frac{1}{2}$. Indeed, the center $C$ of the intersection of $B_{3}(t)$ and the hyperplane $\mathcal{I}:=\{(y_1, y_2, y_3)\in \mathbb{R}^3; y_1+y_2+y_3=1\} \supset \mathcal{H}$ is $C=(\frac{1}{3}, \frac{1}{3}, \frac{1}{3}).$ Thus,
the distance from the origin to  $\mathcal{I}$ is $d=((\frac{1}{3})^2+(\frac{1}{3})^2+(\frac{1}{3})^2)^{1/2} = \frac{1}{\sqrt{3}}.$
By Pythagorean's theorem, the radius of the intersection (disc) on $\mathcal{I}$ is
$$R(t) = \sqrt{t - d^2} = \sqrt{t - \frac{1}{3}}.$$
Let $t_0$ be the value such that the intersection $B_3(t)\cap \mathcal{H}$ exactly inscribes $\mathcal{H}$. By symmetry, the intersection point at the $(x,y)$-plane is
$M = (\frac{1}{2},\frac{1}{2}, 0)$; see Figure \ref{fig:region}(b). Therefore $|CM| = \sqrt{\frac{1}{6}}.$
Solving $t_0$ from  $|CM| = R(t_0),$ we have $t_0 = \frac{1}{2}$.

\begin{figure}[h!]
  \centering
\includegraphics[width=11cm]{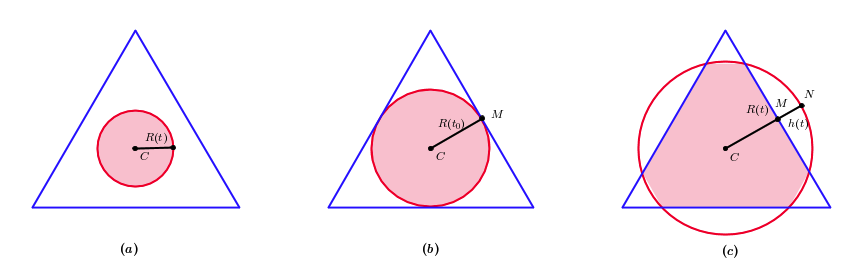}
\caption{Shaded region indicates volume $V(t)$ of intersection as $t$ changes from $1/3$ to 1 as $m=3$.}
\label{fig:region}
\end{figure}

\medskip

When $\frac{1}{3} \le t < \frac{1}{2}$, the intersection locates entirely in $\mathcal{H}$; see Figure \ref{fig:region}(a). Then
$$V(t) = \pi R(t)^2 = \pi (t - \frac{1}{3}).$$ When $\frac{1}{2} \le t \le 1$, the volume of the intersection part [see Figure \ref{fig:region}(c)] is given by
$$V(t) = \pi R(t)^2 - 3 \cdot V_{\text{cs}}(h(t),R(t)),$$
where $V_{\text{cs}}(h(t),R(t))$ is the area of circular segment with radius $R(t)$ and height
$$h(t)=R(t)- |CM|= \sqrt{t-\frac{1}{3}}-\sqrt{\frac{1}{6}}.$$ Therefore, it is easy to check
$$V(t)= \pi (t - \frac{1}{3}) - 3(t - \frac{1}{3}) \arccos \frac{1}{\sqrt{6t-2}} + 3 \sqrt{\frac{1}{6}\big(t-\frac{1}{2}\big)}\,.$$
This and \eqref{eq:dist} yield the desired conclusion.


\medskip

\noindent\textbf{Acknowledgements}. We thank Professors Valentin F\'{e}ray, Sho Matsumoto and Andrei Okounkov very much for communications and discussions. We thank the anonymous referee for the careful reading of our manuscript and many insightful comments and suggestions. 





\end{document}